\documentclass{amsart}
\usepackage{amsmath,amsthm}
\usepackage{amsfonts,amssymb}
\usepackage{hyperref}
\usepackage{enumerate}

\newtheorem{theorem}{Theorem}
\newtheorem{Prop}[theorem]{Proposition}
\newtheorem{lemma}[theorem]{Lemma}

\numberwithin{theorem}{section}

\newtheorem{Cor}[theorem]{Corollary}

\theoremstyle{definition}
\newtheorem{Def}[theorem]{Definition}
\newtheorem{Remark}[theorem]{Remark}

\theoremstyle{remark}

\usepackage{appendix} 
 
\usepackage{tikz}
\usepackage{tikz-cd}
\usepackage{graphicx}
\usepackage{comment}
\usepackage{float}
\usepackage{caption}
\usepackage{color}
\usepackage{tensor}

\usetikzlibrary{shapes.geometric}
\usetikzlibrary{matrix, calc, arrows}

\usepackage{ dsfont }
\usepackage{ mathrsfs }
\usepackage{ytableau}
\usepackage{mathtools}

\def\R{\mathds{R} }
\def\E{\mathds{E} }
\def\Z{\mathds{Z} }
\def\N{\mathds{N} }
\def\C{\mathds{C} }

\def\K{\mathds{K} }

\DeclareMathOperator{\TKK}{TKK}
\DeclareMathOperator{\ad}{ad}
\DeclareMathOperator{\Inn}{Inn}
\DeclareMathOperator{\End}{End}
\DeclareMathOperator{\bessel}{\mathcal B_\lambda}

\DeclareMathOperator{\Fock}{\mc F_\lambda}

\DeclareMathOperator{\HypFdegen}{{}_2F_0}
\DeclareMathOperator{\HypFpq}{{}_pF_q}
\DeclareMathOperator{\KU}{U}
\DeclareMathOperator{\KO}{\widetilde U}
\DeclareMathOperator{\KV}{\widetilde V}

\DeclareMathOperator{\Rsq}{R^2_\lambda}

\DeclareMathOperator{\SB}{\pil(C)}
\newcommand{\iSB}{\pil(C)^{-1}}

\DeclareMathOperator{\rSB}{SB}
\newcommand{\irSB}{\rSB^{-1}}

\newcommand{\bfip}[1]{\left<{#1}\right>_\mathcal B}
\newcommand{\bfipx}[1]{\left<{#1}\right>_{\mathcal B(x)}}

\newcommand{\bfipbar}[1]{\overline{\left<{#1}\right>}_\mathcal B}
\newcommand{\ip}[1]{\left<{#1}\right>_W}

\newcommand{\pt}[1]{\partial_{#1}}
\newcommand{\mf}[1]{\mathfrak{#1}}
\newcommand{\ds}[1]{\mathds{#1}}
\newcommand{\mc}[1]{\mathcal{#1}}

\newcommand{\pushright}[1]{\ifmeasuring@#1\else\omit\hfill$\displaystyle#1$\fi\ignorespaces}
\newcommand{\ol}[1]{\overline{#1}}

\newcommand{\g}{{\mathfrak{g}}}
\newcommand{\mg}{{\mathfrak{g}}}
\newcommand{\oa}{\bar{0}}
\newcommand{\ob}{\bar{1}}

\newcommand{\dea}{{\delta_1}}
\newcommand{\deb}{{\delta_2}}
\newcommand{\dec}{{\delta_3}}
\newcommand{\minus}{\scalebox{0.9}{{\rm -}}}
\newcommand{\plus}{\scalebox{0.6}{{\rm+}}}

\DeclareMathOperator{\Der}{Der}
\DeclareMathOperator{\pil}{\pi_\lambda}
\DeclareMathOperator{\rol}{\rho_\lambda}
\newcommand{\Ila}{\mc I_\lambda}
\newcommand{\Ial}{\mc I_\alpha}

\DeclareMathOperator{\Ad}{Ad}

\DeclarePairedDelimiter\abs{\lvert}{\rvert}%
\DeclarePairedDelimiter\norm{\lVert}{\rVert}%
\makeatletter
\let\oldabs\abs
\def\abs{\@ifstar{\oldabs}{\oldabs*}}
\let\oldnorm\norm
\def\norm{\@ifstar{\oldnorm}{\oldnorm*}}
\makeatother

\begin{document}
\title[Schr\"odinger model, Fock model and intertwiner for $D(2,1;\alpha)$]{A Schr\"odinger model, Fock model and intertwining Segal-Bargmann transform for the exceptional Lie superalgebra $D(2,1;\alpha)$}

\author{Sigiswald Barbier}
\address{Department of Electronics and Information Systems \\Faculty of Engineering and Architecture\\Ghent University\\Krijgslaan 281, 9000 Gent\\ Belgium.}
\email{Sigiswald.Barbier@UGent.be}

\author{Sam Claerebout}
\address{Department of Electronics and Information Systems \\Faculty of Engineering and Architecture\\Ghent University\\Krijgslaan 281, 9000 Gent\\ Belgium.}
\email{Sam.Claerebout@UGent.be}

\date{\today}

\begin{abstract}
We construct two infinite-dimensional irreducible representations for $D(2,1;\alpha)$: a Schrödinger model and a Fock model. Further, we also introduce an intertwining isomorphism. These representations are similar to the minimal representations constructed for the orthosymplectic Lie supergroup and for Hermitian Lie groups of tube type.    The intertwining isomorphism is the analogue of the Segal-Bargmann transform for the orthosymplectic Lie supergroup and for Hermitian Lie groups of tube type.    
\end{abstract}

\maketitle

\tableofcontents

\section{Introduction}
An important theme in representation theory is a good understanding of the class of all unitary representation of a given Lie group. One way to approach this problem is via the orbit method. This method developed by Kirillov, Kostant, Duflo, Vogan and many others \cite{Kirillov}, says that there should be a correspondence between the irreducible unitary representations of a Lie group and the orbits of its dual Lie algebra under the coadjoint action. This connection is one-on-one for nilpotent Lie groups. For other Lie groups this is no longer true, but it is still believed that these coadjoint orbits should still be linked to unitary irreducible representations. 

Under this orbit philosophy, the minimal representation of a semisimple Lie group is the one that ought to correspond to the minimal nilpotent orbit, (see \cite{GanSavin} for a precise definition of a minimal representation). The standard tools of the orbit method seem to work the least for this minimal representation, but this makes it study even more interesting. The fact that the minimal representation	 is in some sense the `smallest' infinite dimensional representation allows for a rich variety of interesting realisations that can be and have been studied using a wide range of technics, \cite{VergneRossi, DvorskySahi, KM, KO1, KO2, KO3}. There exists a unified approach to construct such minimal representations using Jordan algebras, \cite{HKMM, HKMO}. 
It is believed that the orbit method should also be useful to understand, study and construct representations for Lie supergroups (or Lie superalgebras, since for technical reasons it is easier to just work on the algebraic level in the super case). For example, also the unitary irreducible representations of nilpotent Lie supergroups can be completely understood using the orbit method, \cite{Salmasian, NeebSalmasian}. Recently, a minimal representation of the orthosymplectic Lie superalgebra has been studied in detail using the Jordan algebra approach, \cite{BC1, BF, BCD}. This minimal representation of $\mathfrak{osp}(p,q|2n)$ can be seen as a super version of the minimal representation of $\mathfrak{o}(p,q)$ which has been studied in detail by Kobayashi and various collaborators, \cite{KO1,KO2,KO3,KM} . 
The goal of this paper is to construct a similar representation for the Lie superalgebra $D(2,1;\alpha)$. In particular we will construct two different models, which are similar to the Schrödinger and Fock model or the orthosymplectic Lie superalgebra studied in \cite{BF, BCD}  and an intertwining isomorphism between these two models, which is similar to the Segal-Bargmann transform studied in \cite{HKMO} for the non-super case and in \cite{BCD} for $\mathfrak{osp}(m,2|2n)$. 

The Lie superalgebra $D(2,1;\alpha)$ is a deformation of $\mathfrak{osp}(4|2)$ depending on a complex parameter $\alpha$, \cite{K, Scheunert}. Together with $G(3)$ and $F(4)$ it belongs to the class of exceptional basic classical Lie superalgebras. In contrast to the basic classical Lie superalgebras of type $A,B,C$ or $D$ these exceptional Lie superalgebras do not have a Lie algebra analogue. This makes the study of these Lie superalgebras more challenging but also intriguing since we do not know what to expect.
Representation theory of the Lie superalgebra $D(2,1;\alpha)$ has already been studied extensively, see for example \cite{Jo, ChengWang, ChenChengLuo}.

The study of infinite dimensional irreducible representations of Lie superalgebras has still a lot of open questions. For example, it is still not clear what a good definition of unitarity should be in the super case. The existing accepted definition \cite[Definition 2 ]{CCTV} has the drawback that a lot of Lie superalgebras do not admit unitary representations at all, \cite[Theorem 6.2.1]{NeebSalmasian}. This is a highly unsatisfying situation, which has inspired many to look for alternative definitions \cite{dGM, Tuynman}. We strongly believe that the construction of explicit models of irreducible representations that `ought' to be unitary such as the ones constructed in this paper should help in this undertaking.

 Let us now take a closer look at the structure of and results in this paper.

\subsection{Contents}
In Section \ref{D(2,1,alpha)}, we introduce $D(2,1;\alpha)$ using a construction of Scheunert and equip it with a three grading coming from a short subalgebra. In Section \ref{Section Jordan}, $D(2,1;\alpha)$ is alternatively constructed as the $\TKK$-algebra of the Jordan superalgebra $D_\alpha$. A $\TKK$-algebra is by construction three graded, and for $D(2,1;\alpha)$ this three grading corresponds to the one considered in Section \ref{D(2,1,alpha)}.  In Section \ref{Section polynomial realisations}, we recall the polynomial realisation constructed in \cite{BC1} for three graded Lie (super)-algebras and apply it to $D(2,1;\alpha)$ to obtain two realisations: a Fock model and a Schrödinger model. The rest of the paper is devoted to the study of these two models and the construction of an intertwining isomorphism: the Segal-Bargmann transform. We start by taking a closer look at the space on which the Fock model is defined in Section \ref{Section Fock space}. In particular, we introduce a Bessel-Fischer product and show that it leads to a non-degenerate superhermitian product. We also have a reproducing kernel on the Fock space. Next we tackle some properties of the Fock model. We show that the Fock representation is skew-symmetric with respect to the Bessel-Fischer product (Theorem \ref{PropSkewSymRho}). We give the branching law of the Fock model for the subalgebras $\mathfrak{osp}(2|2)\oplus \R$ and $\mathfrak{osp}(1|2)$ (Theorem \ref{ThDecF}). We also calculate the Gelfand-Kirillov dimension (Proposition \ref{Prop Gelfand-Kirillov dimension}) of the Fock representation.
In Section \ref{Section SB transform}, we introduce an operator $\pi_\lambda(C^{-1})$ from the Fock model to the Schrödinger model. We show that monomials of the Fock space gets mapped under $\pi_\lambda(C^{-1})$ to confluent hypergeometric functions of the second kind multiplied with an exponential (Theorem \ref{CorMon}).
We define the Segal-Bargmann transform roughly as the inverse of $\pi_\lambda(C^{-1})$. Since the Segal-Bargmann transform is an intertwining isomorphism (Proposition \ref{Prop inverse SB transform}), this immediately leads to a branching law for the Schrödinger model (Theorem \ref{ThDecW}).
We finish the paper by using the representational framework to recover recurrence relations for the confluent hypergeometric functions of the second kind.
\subsection{Notations}
Let us finish the introduction by mentioning some notations used in this paper. We will work over the field $\ds K$, which is either the field of real numbers $\ds R$ or the field of complex numbers $\ds C$. Function spaces will always be defined over $\C$. We use the convention $\ds N = \{0,1,2,\ldots\}$ and denote the complex unit by $\imath$.

A super-vector space is defined as a $\Z_2$-graded vector space, i.e., $V=V_{\oa}\oplus V_{\ob}$, with $V_{\oa}$ and $V_{\ob}$ vector spaces. An element $v$ of a super-vector space $V$ is called homogeneous if it belongs to $V_i$, $i\in \Z_2$. We call $i$ the parity of $v$ and denote it by $|v|$. An homogeneous element $v$ is even if $|v|=0$ and odd if $|v|=1$. When we use $|v|$ in a formula, we are considering homogeneous elements, with the implicit convention that the formula has to be extended linearly for arbitrary elements. We denote the super-vector space $V$ with $V_{\oa} = \ds K^m$ and $V_{\ob} = \ds K^n$ as $\ds K^{m|n}$. 

\section{The Lie superalgebra $D(2,1;\alpha)$} \label{D(2,1,alpha)}
\subsection{The construction of $D(2,1;\alpha)$}

There is a one-parameter family of $17$-dimensional Lie superalgebras of rank $3$ which are deformations of $D(2,1)=\mathfrak{osp}(4|2)$. These Lie superalgebras can be defined using a construction of Scheunert. We will use the notations of \cite{Mu}, where also more details can be found. 

Let $V$ be a two dimensional vector space with basis $u_{+}$ and $u_-$. Let $\psi$ be a non-degenerate skew-symmetric bilinear form with $\psi(u_+,u_-)=1$. Consider $\mathfrak{sl}(V)= \mathfrak{sp}(\psi)$ the algebra of linear transformations preserving $\psi$. Denote by $(V_i, \psi_i)$, $i=1,2,3$,  three copies of $(V,\psi)$. 

We will use the following definition of a Lie superalgebra to define $D(2,1;\alpha)$.
\begin{Def}[Lie superalgebra]
A super-vector space $\g=\g_{\oa}\oplus \g_{\ob}$ is a Lie superalgebra if
\begin{enumerate}
\item $\g_{\oa}$ is a Lie algebra.
\item $\g_{\ob}$ is a $\g_{\oa}$-module.
\item There exists a $\g_{\oa}$-morphism $p:S^2(\g_{\ob})\rightarrow \g_{\oa}$, with $S^2(\g_{\ob})$ the symmetric tensor power.
\item For all $a,b,c \in \g_{\ob}$ the morphism $p$ satisfies \[ [p(a,b),c]+[p(b,c),a]+[p(c,a),b]=0, \] where we denoted the $\g_{\oa}$-action on $\g_{\ob}$ by $[\cdot,\cdot]$. \label{Jacobi identity}
\end{enumerate}
\end{Def}

Set 
\[
\g_{\oa} = \mathfrak{sp}(\psi_1) \oplus \mathfrak{sp}(\psi_2)\oplus \mathfrak{sp}(\psi_3)
\]
and
\[
\g_{\ob}= V_1 \otimes V_2 \otimes V_3.
\]
The action of $\g_{\oa}$ on $\g_{\ob}$ is given by the outer tensor product:
\[
(A,B,C)\cdot  x\otimes y \otimes z = A (x)\otimes y \otimes z+ x\otimes B(y) \otimes z+x\otimes y \otimes C(z).
\]
Define $p_i\colon V_i \times V_i \to \mathfrak{sp}(\psi_i)$ by
\[ p_i(x,y)z= \psi_i(y,z) x- \psi_i(z,x)y.\] 
  For $\sigma_i \in \ds K$ we define the $\g_{\oa}$-morphism $p$ by 
 \begin{align*}
  p(x_1\otimes x_2\otimes x_3,y_1\otimes y_2\otimes y_3)&= \sigma_1 \psi_2(x_2,y_2)\psi_3(x_3,y_3) p_1(x_1,y_1) \\
  &\quad +\sigma_2 \psi_3(x_3,y_3)\psi_1(x_1,y_1) p_2(x_2,y_2)\\
  &\quad +\sigma_3 \psi_1(x_1,y_1)\psi_2(x_2,y_2) p_3(x_3,y_3).
  \end{align*}
The morphism $p$ satisfies condition (\ref{Jacobi identity}) in the definition of a Lie superalgebra if and only if $\sigma_1+\sigma_2+\sigma_3=0$, see \cite[Lemma 4.2.1]{Mu}.

So in that case the algebra $\Gamma(\sigma_1,\sigma_2,\sigma_3)= \g_{\oa}\oplus\g_{\ob}$ is a Lie superalgebra. 
We have \[ \Gamma(\sigma_1,\sigma_2,\sigma_3)\cong \Gamma(\sigma_1',\sigma_2',\sigma_3')
\]
if and only if there is a non-zero scalar $c$ and a permutation $\pi$ of $(1,2,3)$ such that $\sigma_i'= c \sigma_{\pi(i)}$, \cite[Lemma 5.5.16]{Mu}. If $\sigma_i=0$ for $i=1,2$ or $3$, then $\Gamma(\sigma_1,\sigma_2,\sigma_3)$ contains an ideal $I$ such that $\Gamma(\sigma_1,\sigma_2,\sigma_3) / I \cong \mathfrak{sl}(V_i),$ as one can easily deduce from the definition of $p$.

Define \[D(2,1;\alpha) := \Gamma\left(\frac{1+\alpha}{2},\frac{-1}{2},\frac{-\alpha}{2}\right). \]
Assume $\alpha \not\in \{-1,0\}$, then $D(2,1;\alpha)$ is simple and $D(2,1;\alpha)\cong D(2,1;\beta)$ if and only if $\beta$ is in the same orbit as $\alpha$ under the transformations $\alpha \mapsto \alpha^{-1}$ and $\alpha \mapsto -1-\alpha$. 

\subsection{Roots of $D(2,1;\alpha)$} 
Consider the following matrix realisations for the basis elements $\{E_i,F_i,H_i\}$ of $\mathfrak{sl}(V_i)$.
\[
E_i= \left(\begin{array}{rr}
0 & 1 \\
0 & 0
\end{array}\right), \quad F_i =  \left(\begin{array}{rr}
0 & 0 \\
1 & 0
\end{array}\right),\quad  H_i=\left(\begin{array}{rr}
1 & 0 \\
0 & -1
\end{array}\right).
\] 
Then the realisation of the vector space $V_i$ is given by
\[
u_+^i = (1,0)^t , \quad u_-^i = (0,1)^t.
\]
In this realisation we have 
\[
p_i(u_+^i,u_+^i)= 2 E_i, \quad p_i(u_+^i,u_-^i)=-H_i, \quad p_i(u_-^i,u_-^i)= -2 F_i.
\]

The Cartan subalgebra of $D(2,1;\alpha)$ is given by $\mathfrak{h}=\langle H_1,H_2,H_3 \rangle.$ If we define $\mathfrak{h}^\ast=\{\dea,\deb, \dec \}$ by \[
\delta_i (H_j) = \delta_{ij},\]
then the even and odd roots are given by
\[
\Delta_{\ol 0} = \{ \pm 2\dea, \pm 2\deb , \pm 2\dec \}, \qquad \qquad \Delta_{\ol 1} = \{ \pm \dea \pm \deb \pm \dec \}.
\]
The corresponding root vectors are \[
X_{2\delta_i} = E_i, \quad X_{-2\delta_i} = F_i, \quad X_{\pm\dea\pm\deb\pm\dec} = u_{\pm}^1 \otimes u_{\pm}^2 \otimes u_{\pm}^3.
\]

Consider the simple root system 
\[
\Pi = \{ 2\deb, \dea-\deb-\dec, 2\dec \},
\]
then the Cartan matrix is given by
\begin{align*}
\begin{pmatrix}
2 & -1 & 0 \\
-1 & 0 & -\alpha \\
0 &-1 & 2
\end{pmatrix},
\end{align*}
see \cite[Sections 4.2 and 5.3.1]{Mu}. For $\alpha=-1$ this Cartan matrix corresponds to the Lie superalgebra $A(1,1)=\mathfrak{psl}(2|2)$. 
Remark that $D(2,1;-1)$ contains an ideal $I$ with $I \cong \mathfrak{psl}(2|2) $ and $D(2,1;-1)/ I \cong \mathfrak{sl}(2)$, see \cite{Se}.  
For the values $\alpha= 1,-2,$ or $-1/2$, $D(2,1;\alpha)$ is isomorphic to $\mathfrak{osp}(4|2)$.
For more information on irreducible representations of $D(2,1;\alpha)$, see \cite{Jo}.

\subsection{Three grading}\label{Three grading}

Consider the short subalgebra  
\begin{align}\label{Eqsl2triple}
\{ X_{2\deb} + X_{2\dec}, H_{2}+ H_{3}, X_{-2\deb}+X_{-\dec} \}.
\end{align}
This subalgebra is isomorphic to $\mathfrak{sl}(2)$ and the decomposition of $D(2,1;\alpha)$ as eigenspaces under ad$(H_{2}+H_{3})$ gives a $3$-grading on $D(2,1;\alpha)$:
\begin{align*} 
\g_{\plus} &= \{ X_{2\dec} ,  X_{2\deb}, X_{-\dea+\deb+\dec} ,X_{\dea+\deb+\dec}\} \\
\g_{\minus} &=\{ X_{-2\dec} ,  X_{-2\deb},X_{\dea-\deb-\dec}, X_{-\dea-\deb-\dec} \}  \\
\g_0 &=  \{   H_1, H_2,H_3, X_{2\dea}, X_{-\dea+\deb-\dec},X_{\dea+\deb-\dec},X_{-2\dea}, X_{\dea-\deb+\dec},X_{-\dea-\deb+\dec} \}.
\end{align*}

Set~$h:=H_{2}+H_{3}\in \mf h\subset\mf g_0$.
For $\alpha \not\in  \{0,-1\}$,  $\mf g_0\cong\mathfrak{osp}(2|2)\oplus\ds K h$, where the ideal $\ds K h$ is the centre of $\mf g_0$ and $\mathfrak{osp}(2|2)$ is simple.

\subsection{Real forms}

The complex Lie superalgebra $\g_\mathds{C} = D_\C(2,1;\alpha)$ has three different real forms, \cite[Theorem 2.5]{Pa}.
\begin{itemize}
\item $\g_{\oa}= \mathfrak{sl}(2,\ds R)\oplus \mathfrak{sl}(2,\ds R) \oplus \mathfrak{sl}(2,\ds R),$
\item $\g_{\oa}=\mathfrak{su}(2)\oplus \mathfrak{su}(2) \oplus \mathfrak{sl}(2,\ds R),$
\item $\g_{\oa}= \mathfrak{sl}(2,\ds R)\oplus \mathfrak{sl}(2,\ds C) $.
\end{itemize}
We will use the real Lie superalgebra corresponding to $\g_{\oa}= \mathfrak{sl}(2,\ds R)\oplus \mathfrak{sl}(2,\ds R) \oplus \mathfrak{sl}(2,\ds R)$. 

\section{The Jordan superalgebra $D_\alpha$}\label{Section Jordan}

From a Lie (super)algebra equipped with a short subalgebra one can construct a Jordan (super)algebra. Conversely if one has a Jordan (super)algebra, we obtain a three graded Lie (super)algebra by means of a TKK-construction, \cite{CK}. 
In the classical case Jordan algebras and these TKK-constructions were a crucial ingredient in the unified approach to construct minimal representations for the corresponding Lie superalgebras \cite{HKM, HKMO}. Similarly, a Jordan superalgebra and its associated TKK-algebra were used in \cite{BF} to construct a minimal representation of $\mf{osp}(p,q|2n)$. In this section, we will introduce the Jordan superalgebra $D_\alpha$ associated with $D(2,1;\alpha)$. Applying the TKK-construction on $D_\alpha$ gives us a realisation of $D(2,1;\alpha)$ and also introduces an interesting subalgebra, namely the structure algebra.

\subsection{Definition}

\begin{Def}
A \textbf{Jordan superalgebra} is a supercommutative superalgebra $J$ satisfying the Jordan identity
\begin{align*}
(-1)^{|x||z|}[L_{x}, L_{yz}]+(-1)^{|y||x|}[L_{y}, L_{zx}]+(-1)^{|z||y|}[L_{z}, L_{xy}]=0 \text{ for all } x,y,z \in J.
\end{align*}
Here the operator $L_x$ is (left) multiplication with $x$ and $[\cdot\,,\cdot]$ is the supercommutator, i.e. $[L_x,L_y] := L_xL_y - (-1)^{|x||y|}L_yL_x$.
\end{Def}

We use the realisation of $D_\alpha$ given in \cite{CK}. Then $D_\alpha$ is a unital Jordan superalgebra of dimension $(2|2)$ with
\[
(D_\alpha)_{\oa} := \ds K e_1 + \ds K e_2 \text{ and } (D_\alpha)_{\ob} := \ds K \xi + \ds K \eta.
\]
The multiplication table is given by
\[
e_i e_i =e_i, \quad e_1e_2=0, \quad e_i \xi = \tfrac{1}{2} \xi, \quad e_i \eta = \tfrac{1}{2}\eta, \quad \xi \eta= e_1 + \alpha e_2.
\]
For $\alpha\neq 0$ the Jordan superalgebra $D_\alpha$ is simple and $D_\alpha \cong D_{\alpha^{-1}}$. Remark that the unit is given by $1=e_1+e_2$.

If $\alpha=-1$, then $D_\alpha \cong JGL(1|1)$, the full linear Jordan superalgebra of $(1|1)\times (1|1)$ matrices with Jordan product 
\[
a \cdot b = \frac{ab + (-1)^{\abs{a}\abs{b}} ba}{2}.
\]
The isomorphism is given by 
\[
\quad  e_1=\left(\begin{array}{c|c}
1 & 0 \\ \hline
0 & 0
\end{array}\right), \quad  e_2=\left(\begin{array}{c|c}
0 & 0 \\ \hline
0 & 1
\end{array}\right), \quad 
\xi= \left(\begin{array}{c|c}
0 & \sqrt{2} \\ \hline
0 & 0
\end{array}\right), \quad \eta =  \left(\begin{array}{c|c}
0 & 0 \\ \hline
\sqrt{2} & 0
\end{array}\right).
\] 

Consider the three grading on $D(2,1;\alpha)=\g_{\minus}\oplus \g_{0}\oplus \g_{\plus}$ introduced in Section \ref{Three grading}. Then 
 $\g_{\plus}$ can be equipped with the structure of a Jordan superalgebra:
\[
x \cdot y = \frac{1}{2}[[x, X_{-2\deb}+X_{-2\dec} ],y], \text{ for } x,y \in \g_{\plus}
\]

\begin{lemma} \label{Inverse TKK}
The Jordan superalgebra $D_\alpha$ is isomorphic to the Jordan superalgebra $\mf g_{\plus}$. An explicit isomorphism is given by
\[
e_1=X_{2\deb}, \quad e_2=X_{2\dec},\quad \xi= X_{-\dea+\deb+\dec}, \quad \eta= 2X_{\dea+\deb+\dec}.
\]
\end{lemma}

\begin{proof}
This follows from a straightforward verification.
\end{proof}

\subsection{The structure algebra $\mathfrak{str}(D_\alpha)$}
\begingroup
\allowdisplaybreaks
For a Jordan superalgebra $J$, we set
\begin{align*}
\mf{str}(J) &:= \langle L_x \mid x\in J \rangle \oplus \Der(J),\\
\mf{istr}(J) &:= \langle L_x \mid x\in J \rangle \oplus \Inn(J),
\end{align*}
with $\Der(J) := \langle D\in \End(J) \mid [D,L_x] = L_{D(x)} \text{ for all } x\in J \rangle$ the space of derivations of $J$ and  $\Inn(J) := \langle [L_x,L_y] \mid x,y\in J \rangle\subseteq \Der(J)$ the space of inner derivations of $J$. If $\alpha\neq -1$, then one can verify that every derivation of $D_\alpha$ is an inner derivation, i.e., we have $\Inn(D_\alpha) = \Der(D_\alpha)$ for $\alpha\neq -1$.

Set \begin{align*}
e_1 &=\left(1,\,0,\,0,\,0\right)^t,& e_2 &=\left(0,\,1,\,0,\,0\right)^t,\\
\xi & = \left(0,\,0,\,1,\,0\right)^t,& \eta &= \left(0,\,0,\,0,\,1\right)^t. \end{align*}
For $\alpha\neq -1$, a matrix realisation of $\mathfrak{istr}(D_\alpha)=\mathfrak{str}(D_\alpha)$ is given by 
\begin{align*}
2L_{e_1} &=\left(\begin{array}{cc|cc}
2 & 0 & 0 & 0 \\
0 & 0 & 0 & 0 \\ \hline
0 & 0 & 1 & 0 \\
0 & 0 & 0 & 1
\end{array}\right), &\quad 
2L_{e_2}&= \left(\begin{array}{cc|cc}
0 & 0 & 0 & 0 \\
0 & 2 & 0 & 0 \\ \hline
0 & 0 & 1 & 0 \\
0 & 0 & 0 & 1
\end{array}\right)
, \\
2L_{\xi} &=\left(\begin{array}{cc|cc}
0 & 0 & 0 & 2 \\
0 & 0 & 0 & 2\alpha \\ \hline
1 & 1 & 0 & 0 \\
0 & 0 & 0 & 0
\end{array}\right), &\quad 
2L_{\eta}&= \left(\begin{array}{cc|cc}
0 & 0 & -2 & 0 \\
0 & 0 & -2\alpha & 0 \\ \hline
0 & 0 & 0 & 0 \\
1 & 1 & 0 & 0
\end{array}\right)
, \\
 4[L_{e_1},L_{\xi}]&=\left(\begin{array}{cc|cc}
0 & 0 & 0 & 2 \\
0 & 0 & 0 & -2\alpha \\ \hline
-1 & 1 & 0 & 0 \\
0 & 0 & 0 & 0
\end{array}\right), &\quad
4[L_{e_1},L_{\eta}]&= \left(\begin{array}{cc|cc}
0 & 0 & -2 & 0 \\
0 & 0 & 2\alpha & 0 \\ \hline
0 & 0 & 0 & 0 \\
-1 & 1 & 0 & 0
\end{array}\right), 
\\
4[L_{\xi},L_{\xi}]&=\left(\begin{array}{cc|cc}
0 & 0 & 0 & 0 \\
0 & 0 & 0 & 0 \\ \hline
0 & 0 & 0 &  4(\alpha + 1) \\
0 & 0 & 0 & 0
\end{array}\right),&\quad
4[L_{\eta},L_{\eta}]&=\left(\begin{array}{cc|cc}
0 & 0 & 0 & 0 \\
0 & 0 & 0 & 0 \\ \hline
0 & 0 & 0 & 0 \\
0 & 0 & - 4( \alpha +1) & 0
\end{array}\right),
\end{align*}\vspace{-1em}\begin{align*}
4[L_{\xi},L_{\eta}]=\left(\begin{array}{cc|cc}
0 & 0 & 0 & 0 \\
0 & 0 & 0 & 0 \\ \hline
0 & 0 & -2(\alpha+1) &  0 \\
0 & 0 & 0 & 2(\alpha+1)
\end{array}\right).
\end{align*}
Note that $[L_{e_1},L_{e_2}]=0$, $[L_{e_1},L_{\xi}]=-[L_{e_2},L_{\xi}]$ and $[L_{e_1},L_{\eta}]=-[L_{e_2},L_{\eta}]$. The matrix realisation of $\mathfrak{istr}(D_{-1})$ is the same but note that 
\[
[L_{\xi},L_{\xi}] = [L_\eta,L_\eta]=[L_\xi,L_\eta]=0.
\]
For the matrix realisation of $\mathfrak{str}(D_{-1})$ we add the following three derivations
\begin{align*}
d_-&:=\left(\begin{array}{cc|cc}
0 & 0 & 0 & 0 \\
0 & 0 & 0 & 0 \\ \hline
0 & 0 & 0 &  1 \\
0 & 0 & 0 & 0
\end{array}\right),&\quad
d_+&:=\left(\begin{array}{cc|cc}
0 & 0 & 0 & 0 \\
0 & 0 & 0 & 0 \\ \hline
0 & 0 & 0 & 0 \\
0 & 0 & 1 & 0
\end{array}\right),\\
d_0&:=\left(\begin{array}{cc|cc}
0 & 0 & 0 & 0 \\
0 & 0 & 0 & 0 \\ \hline
0 & 0 & -1 &  0 \\
0 & 0 & 0 & 1
\end{array}\right)
\end{align*}
to the matrix realisation of $\mathfrak{istr}(D_{-1})$.
\endgroup
\subsection{The TKK-construction for $D_\alpha$}
With each Jordan (super)algebra one can associate a $3$-graded Lie (super)algebra via the TKK-construction. There exist different TKK-constructions in the literature, see \cite{BC2} for an overview. For $D_\alpha$ with $\alpha\neq -1$ all constructions lead to $D(2,1;\alpha)$, while for $D_{-1}$ we either get $\mathfrak{psl}(2|2)$ or $D(2,1;-1)$.
As a super-vector space $\TKK(D_\alpha)$ is given by  $\TKK(D_\alpha)= D_\alpha^{\minus} \oplus \mathfrak{str} (D_\alpha) \oplus D_\alpha^{\plus}$ with $ D_\alpha^{\minus}$ and $ D_\alpha^{\plus}$ two copies of $ D_\alpha$. For $x,y \in D_\alpha^{\plus}$, $u,v\in D_\alpha^{\minus}$, $a,b \in D_\alpha$, $I\in \Der(D_\alpha)$ the Lie brackets are defined as
\begin{align*}
[L_a, x] &= ax \in  D_\alpha^{\plus} & [I,x]&=Ix \in  D_\alpha^{\plus}\\
[L_a, u] &= -au \in  D_\alpha^{\minus} & [I,u]&=Iu \in  D_\alpha^{\minus}\\
[x,u]&=2L_{xu}+2[L_x,L_u] \\
[x,y]&=0=[u,v],
\end{align*}
and embed $\mathfrak{str}(D_\alpha)$ as a subalgebra into $\TKK(D_\alpha)$. 

We will use the notation $f_1, f_2, \zeta$ and $\theta$ instead of $e_1,e_2, \xi$ and $\eta$ for the generators of $D_\alpha^{\minus}$. Then, for $\alpha\not= -1$, an explicit morphism between $\TKK(D_\alpha)$ and the realisation of $D(2,1;\alpha)$ given in Section \ref{D(2,1,alpha)} is as follows.
\begin{itemize}
\item For $D_\alpha^{\minus}$
\begin{align*}
f_1= X_{-2\deb}, \quad f_2= X_{-2\dec}, \quad \zeta=  -X_{-\dea- \deb -\dec} \quad \theta=  -2X_{\dea-\deb-\dec} .
\end{align*}
\item For $D_\alpha^{\plus}$
\begin{align*}
e_1= X_{2\deb}, \quad e_2= X_{2\dec}, \quad \xi =  X_{-\dea+ \deb +\dec} \quad \eta=  2X_{\dea+\deb+\dec} .
\end{align*}
\item For $\mf{str}(D_\alpha)$
\begin{align*}
2L_{e_1} &= H_{2}, & 2L_{e_2} &=H_{3},\\
2L_\xi &= -X_{-\delta_1-\delta_2+\delta_3}-X_{-\delta_1+\delta_2-\delta_3}, & 2L_\eta &= -2(X_{\delta_1-\delta_2+\delta_3}+X_{\delta_1+\delta_2-\delta_3}),\\
4[L_{e_1}, L_\xi] &= X_{-\delta_1-\delta_2+\delta_3}-X_{-\delta_1+\delta_2-\delta_3}, & 4[L_{e_1}, L_\eta] &= 2(X_{\delta_1-\delta_2+\delta_3}-X_{\delta_1+\delta_2-\delta_3}),\\
4[L_\xi, L_\xi] &= 2(1+\alpha)X_{-2\delta_1}, & 4[L_\eta, L_\eta] &= -8(1+\alpha)X_{2\delta_1},\\
4[L_\xi, L_\eta] &= 2(1+\alpha)H_1.
\end{align*}
\end{itemize}

An explicit morphism between $\TKK(D_{-1})$ and $D(2,1;-1)$ is given by the same isomorphism and for the extra derivations $d_-, d_+, d_0$ we have
\begin{align*}
d_- &= X_{-2\delta_1}, & d_+ &= X_{2\delta_1},\\
d_0 &= H_1.
\end{align*}

An alternative TKK-construction is given by $\overline{\TKK}(D_\alpha)= D_\alpha^{\minus} \oplus \mathfrak{istr} (D_\alpha) \oplus D_\alpha^{\plus}$. For $\alpha \neq-1$ this is the same construction. For $\alpha=-1$ we now have $\overline{\TKK}(D_{-1}) \cong \mathfrak{psl}(2|2)$, with the same isomorphism as for $D(2,1;-1)$, but without $d_-$, $d_+$ and $d_0$. 

\begin{Remark}
Although $D(2,1;-1)$ and $D(2,1;0)$ are isomorphic and both have $\mathfrak{psl}(2|2)$ as a unique ideal, this isomorphism does not respect the three grading introduced in Section \ref{Three grading}.  In particular the top  of $D(2,1;-1)$ (i.e. the module  we obtain by quotienting out $\mathfrak{psl}(2|2)$) is contained in the zero part of the three grading, while the top of $D(2,1;0)$ is spread over the whole three grading. So one has to be careful in identifying these cases. For example, the corresponding Jordan superalgebras $D_{-1}$ and $D_{0}$ are not isomorphic (which follows easily from the fact that $D_{-1}$ is simple and $D_{0}$ is not.) Furthermore while $\mf{istr}(D_{-1})\not= \mf{str} (D_{-1})$ we have $\mf{istr}(D_{0})= \mf{str} (D_{0})$.     
\end{Remark}

\section{Two polynomial realisations}
\label{Section polynomial realisations}
In \cite{BC1} an explicit polynomial realisation for $D(2,1;\alpha)$ was constructed that is an analogue of the conformal representations considered in \cite{HKM}. In this section we repeat this construction and use it to define two models of a representation of $D(2,1;\alpha)$: the Fock model and the Schr\"odinger model. We start this section by introducing the Bessel operators which will play a crucial role in the polynomial realisation.
 
\subsection{The Bessel operator}

Let $\lambda$ be a character of $\mathfrak{str}(D_\alpha)$. Then the Bessel operator $\bessel$ maps an element of $D_\alpha^{\plus}$ to a differential operator acting on $S(D_\alpha^{\minus})$, the supersymmetric tensor power of $D_\alpha^{\minus}$. To define this Bessel operator we need the following two operators.

\begin{Def}
Consider a  character $\lambda\colon \mathfrak{str}(D_\alpha) \to \ds K$.  We define $\lambda_{u} \in (D_\alpha^{\plus})^\ast$ and $\widetilde{P}_{u,v} \in D_\alpha^{\minus}\otimes (D_\alpha^{\plus})^\ast $ for  $u,v$ in $D_\alpha^{\minus}$ by 
$$\lambda_{u}(x) := -\lambda(2L_{xu})$$
and
\[\widetilde{P}_{u,v}(x) := (-1)^{\abs{x}(\abs{u}+\abs{v})}(L_u L_v + (-1)^{|u||v|}L_vL_u-L_{uv})(x)\]
for all $x$ in $D_\alpha^{\plus}$. Then we define the \textbf{Bessel operator} as
\begin{align}\label{EqBess}
\bessel = \sum_{i=1}^4 \lambda_{f_i}\pt{f_i}+ \sum_{i,j=1}^4 \widetilde{P}_{f_i,f_j}\pt{f_j}\pt{f_i},
\end{align}
with $(f_i)_{i=1}^4$ a homogeneous basis of $D_\alpha^{\minus}$.
\end{Def}

Recall that for $\alpha\not\in \{0, -1\}$ we have $\mathfrak{str}(D_\alpha) \cong \mathfrak{osp}(2|2) \oplus \ds K h$ with $h = H_{2} + H_{3} = 2L_{e_1}+ 2L_{e_2}$. Therefore the character $\lambda$ is determined by the value of $\lambda(h)$. We  normalize this value as  $\lambda(h)= \frac{\alpha+1}{\alpha}\lambda $, with $\lambda \in \mathds{C}$. Then $\lambda(H_{2})= \lambda$ and $\lambda(H_{3})= \frac{\lambda}{\alpha}$ since $\lambda(H_{2}+H_{3})=\lambda(h)=\lambda \frac{\alpha+1}{\alpha}$ and $\lambda(H_{2}-\alpha H_{3})=0$. For $\alpha=0$ one verifies that $\lambda(H_2)=0$ and and the character is determined by $\lambda := \lambda(h) = \lambda(H_3)$. Similarly for $\alpha=-1$ we have $\lambda(H_{2}) = -\lambda(H_{3})$ and this value determines the character.  If we then set $\lambda := \lambda(H_{2})$, the results are analogous to the $\alpha\not\in \{-1,0\}$ case. 

For each basis element $x$ of $D_\alpha$ consider the dual element $x^\ast \in D_\alpha^\ast$ defined by sending $x$ to one and all other basis elements to zero.
\begin{lemma}\label{Lem Bessel}
We have
\begin{align*}
\lambda_{f_1} &= -\lambda e_1^\ast, \quad \lambda_{f_2} = \frac{-\lambda}{\alpha} e_2^\ast, \quad \lambda_{\zeta}=2 \lambda \eta^\ast, \quad \lambda_\theta = -2 \lambda \xi^\ast &\text{ if }\alpha\neq 0,\\
\lambda_{f_1} &= 0, \quad \lambda_{f_2} = -\lambda e_2^\ast, \quad \lambda_{\zeta}=0, \quad \lambda_\theta = 0 &\text{ if }\alpha = 0
\end{align*}
and 
\begin{gather*}
\widetilde{P}_{f_1,f_1 } = f_1 e_1^\ast, \quad  \widetilde{P}_{f_1,f_2 } = \frac{\zeta}{2} \xi^\ast+ \frac{\theta}{2} \eta^\ast, \quad \widetilde{P}_{f_2,f_2 } = f_2 e_2^\ast,  \\
\widetilde{P}_{f_1,\zeta } = \frac{\zeta}{2} e_1^\ast- f_1 \eta^\ast, \quad
\widetilde{P}_{f_1,\theta } = \frac{\theta}{2} e_1^\ast+ f_1 \xi^\ast, \\
\widetilde{P}_{f_2,\zeta } = \frac{\zeta}{2} e_2^\ast- \alpha f_2 \eta^\ast, \quad
\widetilde{P}_{f_2,\theta } = \frac{\theta}{2} e_2^\ast+ \alpha f_2 \xi^\ast, \\
\widetilde{P}_{\zeta,\theta } = \alpha f_2 e_1^\ast + f_1 e_2^\ast-(1+\alpha) \zeta \xi^\ast -(1+\alpha) \theta \eta^\ast \\
\widetilde{P}_{\zeta,\zeta } =0, \quad \widetilde{P}_{\theta,\theta }=0.
\end{gather*}
\end{lemma}

\begin{proof}
This follows from straightforward calculations.
\end{proof}

Using Lemma \ref{Lem Bessel} the Bessel operator (\ref{EqBess}) becomes
\begin{align*}
\mc B_\lambda &=-\lambda e_1^\ast \partial_{f_1}- \frac{\lambda}{\alpha} e_2^\ast\partial_{f_2}
+ 2 \lambda \eta^\ast\partial_{\zeta}- 2 \lambda \xi^\ast \partial_{\theta} \\
&\quad + f_1 e_1^\ast\partial_{f_1}\partial_{f_1 }+ (\zeta \xi^\ast+ \theta \eta^\ast)\partial_{f_2}\partial_{f_1}+ f_2 e_2^\ast\partial_{f_2}\partial_{f_2 } \\
&\quad +(\zeta e_1^\ast- 2f_1 \eta^\ast)\partial_{\zeta}\partial_{f_1 }
+( \theta e_1^\ast+2 f_1 \xi^\ast)\partial_{\theta}\partial_{f_1 } \\
&\quad +(\zeta e_2^\ast- 2 \alpha f_2 \eta^\ast)\partial_{\zeta}\partial_{f_2 }
+( \theta e_2^\ast+2 \alpha f_2 \xi^\ast)\partial_{\theta}\partial_{f_2 } \\
&\quad +2\left(\alpha f_2 e_1^\ast + f_1 e_2^\ast-(1+\alpha) \zeta \xi^\ast -(1+\alpha) \theta \eta^\ast\right)\partial_{\theta}\partial_{\zeta },
\end{align*}
for $\alpha\neq 0$. For $\alpha=0$ we have to replace occurences of $\frac{\lambda}{\alpha}$ by $\lambda$  and other occurences of   $\lambda$ by zero in this expression.  We obtain the following property of the Bessel operator from \cite[Proposition 4.2]{BC1}.

\begin{Prop}[Supercommutativity]
For all $a,b\in D_\alpha^+$ we have
\begin{align*}
\bessel(a)\bessel(b) = (-1)^{|a||b|}\bessel(b)\bessel(a).
\end{align*}
\end{Prop}

\subsection{A polynomial realisation}

Recall that we denote the complex unit by $\imath$. In Section 3 in \cite{BC1}, a polynomial representation of a three graded Lie superalgebra $\mg$ on $S(\mg_{\minus})$ depending on a character $\lambda$ of $\mg_0$ was constructed:
\begin{align*}
\pi(X)&= X,  &\text{ for } X \in \mg_{\minus}, \\
\pi(X)&= \lambda(X) + \sum_{i} [X ,X_i] \partial_{X_i},  &\text{ for } X \in \mg_0, \\
\pi(X) &= -\bessel (X), &\text{ for } X \in \mg_{\plus}. 
\end{align*}
Here $(X_i)_i$ is a homogeneous basis of $\mg_{\minus}$.

Applying this construction to the three grading of $D(2,1;\alpha)$ considered in Section \ref{Three grading} we obtain a representation of $D(2,1;\alpha)$ on $\mc S(f_1,{f_2},\zeta,\theta)\cong \mc P(\K^{2|2})$. Note that we applied an isomorphism induced by $f_1 \mapsto -2\imath f_1$, $f_2 \mapsto -2\imath f_2$, $\zeta \mapsto -2\imath \zeta$, $\theta \mapsto -2\imath \theta$ in order to make our representation more comparable to the representations found in \cite{HKMO} and \cite{BCD}.

For $\mf g_-$ we have
\begin{align*}
\pil(f_1) &= -2\imath f_1, &\pil(f_2) &= -2\imath f_2,\\
\pil(\zeta) &= -2\imath\zeta, &\pil(\theta) &= -2\imath\theta.
\end{align*}
For $\mf g_0$ we have
\begin{align*}
\pil(2L_{e_1}) &=  \lambda -2f_1\pt {f_1} - \zeta \pt \zeta - \theta\pt \theta,\\
\pil(2L_{e_2}) &=  \dfrac{\lambda}{\alpha} -2{f_2}\pt {f_2} - \zeta \pt \zeta - \theta\pt \theta,\\
\pil(2L_{\xi}) &=  -\zeta(\pt {f_1}+\pt {f_2}) - 2(f_1+\alpha {f_2})\pt \theta,\\
\pil(2L_{\eta}) &=  -\theta(\pt {f_1}+\partial_{f_2}) + 2(f_1+\alpha {f_2})\pt \zeta ,\\
\pil(4[L_{e_1},L_{\xi}]) &=  -\zeta(\pt {f_1}-\pt {f_2})+2(f_1 - \alpha {f_2})\pt \theta  ,\\
\pil(4[L_{e_1},L_{\eta}]) &=  -\theta(\pt {f_1} -\pt {f_2})- 2(f_1-\alpha {f_2})\pt \zeta,
\end{align*}
\begin{align*}
\pil(4[L_{\xi},L_{\xi}]) &= 4(1+\alpha)\zeta\pt \theta &\text{ or }\quad &\pil(d_-) = \zeta\pt \theta,\\
\pil(4[L_{\eta},L_{\eta}]) &=  -4(1+\alpha)\theta\pt \zeta &\text{ or }\quad &\pil(d_+) = \theta\pt \zeta,\\
\pil(4[L_{\xi},L_{\eta}]) &=  -2(1+\alpha)(\zeta\pt \zeta -\theta\pt \theta) &\text{ or }\quad &\pil(d_0) = \theta\pt \theta - \zeta\pt \zeta.
\end{align*}
For $\mf g_+$ we have
\begin{align*}
\pil(e_1) &= -\dfrac{\imath}{2}((-\lambda+f_1\pt {f_1} + \zeta\pt \zeta + \theta\pt \theta)\pt {f_1} -2\alpha {f_2} \pt \zeta \pt \theta), \\
\pil(e_2) &= -\dfrac{\imath}{2} ((-\dfrac{\lambda}{\alpha}+{f_2}\pt {f_2} + \zeta\pt \zeta + \theta\pt \theta)\pt {f_2} -2 f_1 \pt \zeta \pt \theta),\\
\pil(\xi) &= -\dfrac{\imath}{2} ((-2\lambda + 2f_1\pt {f_1} + 2\alpha {f_2}\pt {f_2} + 2(1+\alpha)\zeta\pt \zeta)\pt \theta + \zeta \pt {f_1} \pt {f_2}),\\
\pil(\eta) &= -\dfrac{\imath}{2} ((2\lambda - 2f_1\pt {f_1} - 2\alpha {f_2}\pt {f_2} - 2(1+\alpha)\theta\pt \theta)\pt \zeta + \theta \pt {f_1} \pt {f_2}).
\end{align*}

Note that we only showed the $\alpha\neq 0$ case. The $\alpha=0$ case is obtained by substituting every instance of $\lambda/\alpha$ by $\lambda$ and every other instance of $\lambda$ by $0$ in the above expressions.

From \cite[Proposition 5.3]{BC1} we obtain the following.

\begin{Prop}\label{propsimple}
Suppose $\alpha \not\in \{-1,0\}$. The representation $\pi_\lambda$ of $D(2,1;\alpha)$ on $\mc P(\ds \K^{2|2})$ is irreducible if and only if 
\[\lambda \not\in \ds N \quad\mbox{and}\quad \dfrac{\lambda}{\alpha}\not\in\ds N.\]
If either $\lambda\in \ds N $ or $\frac{\lambda}{\alpha}\in\ds N,$ the representation is indecomposable but not irreducible.
\end{Prop}
As discussed in \cite[Section 5]{BC1} the representations that ought to lead to `minimal representations' are the ones corresponding to quotients of $\pi_\lambda$ for the specific cases $\lambda=1$ and $\lambda=\alpha$. Let us consider these cases in more detail.
First assume $\alpha\neq 1$. An easy calculation shows that for $\alpha \neq 0$ the Bessel operators act trivially on 
\begin{align*}
V_\alpha&:=\{ a(2f_1f_2+\zeta\theta) + b{f_2^2} + c {f_2}\zeta +d {f_2}\theta \mid a,b,c,d \in \mathds{K} \} \subset \mathcal{P}_2 & \text{if } \lambda&=\alpha \text{ and }\\
V_1&:=\{ a (2\alpha f_1f_2+\zeta\theta) + b f_1^2 + c f_1\zeta + d f_1\theta \mid a,b,c,d \in \mathds{K} \} \subset \mathcal{P}_2 & \text{if } \lambda&=1.
\end{align*}
For $\alpha=0$ the Bessel operators act trivially on 
\begin{align*}
V_1 &:=\{ a(2f_1f_2+\zeta\theta) + b{f_2^2} + c {f_2}\zeta +d {f_2}\theta \mid a,b,c,d \in \mathds{K} \} \subset \mathcal{P}_2 & \text{if } \lambda&=1. \\
V_0 &:=\{ a f_1+ b f_2 +c \zeta + d \eta \mid a,b,c,d \in \mathds{K} \} = \mathcal{P}_1
& \text{ if } \lambda &=\alpha=0.  
\end{align*}

Furthermore, $V_\alpha$ and $V_1$ are $\mg_0$-modules.
It then follows from the Poincaré-Birkhoff-Witt theorem that
\begin{align}\label{Eq I lambda}
\mc I_\lambda := U(\mg_{\minus}) V_\lambda=  \mc{P}(\ds \K^{2|2}) V_\lambda
\end{align} is a submodule of $\pi_\lambda$ for $\lambda=\alpha$ or $\lambda=1$.  
In the spirit of \cite{HKMO} and \cite{BC1} the quotient representation of $D(2,1;\alpha)$ on $\mc{P} (\mathds{K}^{2 |2}) / \mc{I}_\lambda$ should be considered the analogue of minimal representation. In the rest of this paper we will take a closer look at this quotient representation.
 
Note that for $\lambda=\alpha=1$ the picture changes. We now have a one-dimensional $\mg_0$-submodule generated by $R^2=2f_1f_2+\zeta\theta$ and the logical quotient representation to study would be $\mc{P} (\mathds{K}^{2 |2}) / \langle R^2\rangle$. Since $D(2,1;1) \cong \mf{osp}(4|2)$ and $D_1$ is isomorphic to the spin factor Jordan superalgebra, this was already studied in a more general setting in \cite{BF} and \cite{BCD}.

\subsection{The Fock representation}\label{Section Cayley}

Consider the (real) Lie superalgebras
\begin{align*}
\mathfrak{g}&:= \TKK(D_{\alpha}) \cong D(2,1;\alpha),\\
\mf k &:= \{(a, I, -a) \mid I \in \Der(D_\alpha), a\in D_\alpha\}\\
\mf k_{0} &:= \{a\in \mf{str}(D_\alpha) \mid [a,e] = 0\} = \mf k\cap \mf{str}(D_\alpha) = \Der(D_\alpha), 
\end{align*}
where $e$ is the unit of $D_\alpha$.
\begin{Prop}
For $\alpha \in \ds R\setminus \{-1\}$ the following (real) algebras are isomorphic:
\begin{align*}
\mf k_0 \cong \mf{osp}(1,0|2), \quad \mf k \cong \mf{osp}(2,0|2)\oplus \R.
\end{align*}
For $\alpha = -1$ we have
\begin{align*}
\mf k_0 \cong \mf{sl}(2)\ltimes (\R\oplus\R), \quad \mf k \cong \mf{str}(D_{-1}) \cong \mf{sl}(2)\ltimes s(\mf{gl}(1|1)\oplus\mf{gl}(1|1))/\left<I_4\right>,
\end{align*}
where the action of $\mf{sl}(2)$ is adjoint action by using the embedding of $\mf{sl}(2)$ in $D(2,1;-1)$.
\end{Prop}
\begin{proof}
A matrix realisation of $\mf{osp}(2,0|2)\oplus \R$ can be given by
\begin{align*}
\mf{osp}(2,0|2)\oplus \R = \left\lbrace \left(\begin{array}{ccc|cc}
a & 0 & 0 & 0 & 0 \\
0 & 0 & b & c & d \\
0 & -b & 0 & e &  f \\ \hline
0 & d & f & g & h\\
0 & -c & -e & i & -g
\end{array}\right)| a,b,c,d,e,f,g,h,i\in \ds R \right\rbrace.
\end{align*}
Denote by $E_{ij}$ the matrix with the $(i,j)$-th entry equal to one and all other entries zero. Then an explicit isomorphism between $\mf k$ and the matrix realisation of $\mf{osp}(2,0|2)\oplus \R$ is given by
\begin{align*}
e_1-f_1 &=E_{11}+E_{23}-E_{32}, & e_2-f_2 &= E_{11}-E_{23}+E_{32}\\
\xi-\zeta &= \sqrt{-2(1+\alpha)}(E_{25}+E_{42}), & \eta-\theta &= \sqrt{-2(1+\alpha)}(E_{52}-E_{24})\\
4[L_{e_1},L_\xi] &= \sqrt{-2(1+\alpha)}(E_{35}+E_{43}), & 4[L_{e_1},L_\eta]&= \sqrt{-2(1+\alpha)}(E_{53}-E_{34})\\
4[L_\xi,L_\xi] &= 4(1+\alpha)E_{45}, & 4[L_\eta,L_\eta]&= -4(1+\alpha)E_{54}\\
4[L_\xi,L_\eta] &= -2(1+\alpha)(E_{44}-E_{55}),
\end{align*}
for $\alpha < -1$. Note that this isomorphism also maps $\mf k_0$ to $\mf{osp}(1,0|2)$. Similarly, for $\alpha>-1$ an explicit isomorphism between $\mf k$ and $\mf{osp}(0,2|2)\oplus \R$ is given by
\begin{align*}
e_1-f_1 &=E_{11}+E_{23}-E_{32}, & e_2-f_2 &= E_{11}-E_{23}+E_{32}\\
\xi-\zeta &= \sqrt{2(1+\alpha)}(E_{42}-E_{25}), & \eta-\theta &= \sqrt{2(1+\alpha)}(E_{24}+E_{52})\\
4[L_{e_1},L_\xi] &= \sqrt{2(1+\alpha)}(E_{43}-E_{35}), & 4[L_{e_1},L_\eta]&= \sqrt{2(1+\alpha)}(E_{34}+E_{53})\\
4[L_\xi,L_\xi] &= 4(1+\alpha)E_{45}, & 4[L_\eta,L_\eta]&= -4(1+\alpha)E_{54}\\
4[L_\xi,L_\eta] &= -2(1+\alpha)(E_{44}-E_{55}),
\end{align*}
which maps $\mf k_0$ to $\mf{osp}(0,1|2)$. Since $\mf{osp}(p,q|2)\cong \mf{osp}(q,p|2)$ for all $p,q\in\ds\N$ this proves the $\alpha\neq -1$ case. For $\alpha=-1$ an explicit isomorphism between $\mf k$ and $\mf{str}(D_{-1})$ is given by
\begin{gather*}
e_1-f_1 =2L_{e_1},\quad e_2-f_2 = 2L_{e_2}, \quad \xi-\zeta = -2L_{\xi}, \quad \eta-\theta = -2L_{\eta}\\
[L_{e_1},L_\xi] = -[L_{e_1},L_\xi], \quad [L_{e_1},L_\eta]= -[L_{e_1},L_\eta], \quad d_- = d_-, \quad d_+ = d_+, \quad d_0 = d_0.
\end{gather*}
The isomorphism $\mf{str}(D_{-1}) \cong \mf{sl}(2)\ltimes s(\mf{gl}(1|1)\oplus\mf{gl}(1|1))/\left<I_4\right>$ follows from \cite[Section 6.1]{BC2}\\
\end{proof}
The representation $\pi_\lambda$ of $\mg$ on $\mc{P} (\mathds{C}^{2 |2}) / \mc{I}_\lambda$ has no $\mf{k}$-finite vectors. We will remedy this by twisting by a Cayley transform $c$. 
So we define our Fock model as the representation $\rho_\lambda := \pi_{\lambda} \circ c$. The analogue of this Fock model is studied in \cite{HKMO} for classical Lie algebras, and in \cite{BCD} for $\mf{osp}(m,2|2n)$.

Let $e := e_1+e_2$ and $f := f_1+f_2$ denote the units of $D_\alpha^+$ and $D_\alpha^-$, respectively and let $h = 2L_e$. Then $\{h, e, f\}$ corresponds to the $\mf{sl}(2)$-triple given in (\ref{Eqsl2triple}). 

Define the Cayley transform $c \in \End(\mf{g}_\C)$ as
\begin{align*}
c := \exp(\frac{\imath}{2}\ad(f))\exp(\imath\ad(e)).
\end{align*}

\begin{Prop} Using the decomposition $\mf{g}_\C= D^-_{\alpha,\C}\oplus \mathfrak{str}(D_{\alpha,\C})\oplus D^+_{\alpha,\C} $ we obtain the following explicit expression for the Cayley transform
\begin{itemize}
\item $c(a,0,0) = \left(\dfrac{a}{4}, \imath L_a, a\right)$
\item $c(0,L_a+I,0) = \left(\imath \dfrac{a}{4}, I, -\imath a\right)$
\item $c(0,0,a) = \left(\dfrac{a}{4}, -\imath L_a, a\right)$,
\end{itemize}
with $a\in D_{\alpha,\C}$ and $I\in \Der(D_{\alpha,\C})$. It induces a Lie superalgebra isomorphism:
\begin{align*}
c:\;\mathfrak{k}_\C \rightarrow \mathfrak{str}(D_{\alpha,\C}),\quad (a,I,-a) \mapsto I + 2\imath L_a.
\end{align*}
\end{Prop}

\begin{proof}
This follows from the same straightforward calculations as given in the proof of \cite[Proposition 5.1]{BCD}.
\end{proof}

An explicit expression for the Fock representation $\rol := \pil\circ \,c$ on $\mc{P} (\mathds{C}^{2 |2}) / \mc{I}_\lambda$ 
is as follows. Let $z_1,z_2$ and $z_3,z_4$ be the even resp. odd representatives of the coordinates on $\mc{P} (\mathds{C}^{2 |2}) / \mc{I}_\lambda$. 
For $\mf g_-$ we obtain
\begin{align*}
\rol(f_1) &= -\dfrac{\imath}{2}(z_1 + \bessel(z_1)) -\frac{\imath}{2}(-\lambda +2z_1\pt {z_1} + z_3 \pt {z_3} + z_4\pt {z_4}),\\
\rol(f_2) &= -\dfrac{\imath}{2}(z_2 + \bessel(z_2)) -\frac{\imath}{2}(-\dfrac{\lambda}{\alpha} +2z_2\pt {z_2} + z_3 \pt {z_3} + z_4\pt {z_4}),\\
\rol(\zeta) &= -\dfrac{\imath}{2}(z_3 + \bessel(z_3)) -\frac{\imath}{2}(z_3\pt {z_1} + 2\alpha z_2\pt {z_4} +z_3\pt {z_2} + 2 z_1\pt {z_4}),\\
\rol(\theta) &= -\dfrac{\imath}{2}(z_4 + \bessel(z_4)) -\frac{\imath}{2}(z_4\pt {z_1} - 2\alpha z_2\pt {z_3} +z_4\pt {z_2} - 2 z_1\pt {z_3}).
\end{align*}
For $\mf g_+$ we have
\begin{align*}
\rol(e_1) &= -\dfrac{\imath}{2}(z_1 + \bessel(z_1)) +\frac{\imath}{2}(-\lambda +2z_1\pt {z_1} + z_3 \pt {z_3} + z_4\pt {z_4}),\\
\rol(e_2) &= -\dfrac{\imath}{2}(z_2 + \bessel(z_2)) +\frac{\imath}{2}(-\dfrac{\lambda}{\alpha} +2z_2\pt {z_2} + z_3 \pt {z_3} + z_4\pt {z_4}),\\
\rol(\xi) &= -\dfrac{\imath}{2}(z_3 + \bessel(z_3)) +\frac{\imath}{2}(z_3\pt {z_1} + 2\alpha z_2\pt {z_4} +z_3\pt {z_2} + 2 z_1\pt {z_4}),\\
\rol(\eta) &= -\dfrac{\imath}{2}(z_4 + \bessel(z_4)) +\frac{\imath}{2}(z_4\pt {z_1} - 2\alpha z_2\pt {z_3} + z_4\pt {z_2} - 2 z_1\pt {z_3}).
\end{align*}
For $\mf g_0$ we have
\begin{align*}
\rol(2L_{e_1}) &=z_1-\bessel(z_1), &\rol(2L_{e_2}) &= z_2-\bessel(z_2),\\
\rol(2L_{\xi}) &= z_3-\bessel(z_3), &\rol(2L_{\eta}) &= z_4-\bessel(z_4),\\
\rol(4[L_{e_1},L_{\xi}]) &=  -z_3\pt {z_1} - 2\alpha z_2\pt {z_4} +z_3\pt {z_2} + 2 z_1\pt {z_4},\\
\rol(4[L_{e_1},L_{\eta}]) &=  -{z_4}\pt {z_1} + 2\alpha z_2\pt {z_3} + z_4\pt {z_2} - 2 z_1\pt {z_3},
\end{align*}
\begin{align*}
\rol(4[L_{\xi},L_{\xi}]) &= 4(1+\alpha)z_3\pt {z_4} &\text{ or }\quad &\rol(d_-) = z_3\pt {z_4},\\
\rol(4[L_{\eta},L_{\eta}]) &= -4(1+\alpha)z_4\pt {z_3} &\text{ or }\quad &\rol(d_+) = z_4\pt {z_3},\\
\rol(4[L_{\xi},L_{\eta}]) &= -2(1+\alpha)(z_3\pt {z_3}  -z_4\pt {z_4}) &\text{ or }\quad &\rol(d_0) = z_4\pt {z_4} - {z_3}\pt {z_3}.
\end{align*}

Once again we only showed the $\alpha\neq 0$ case. The $\alpha=0$ case is still obtained by substituting every instance of $\lambda/\alpha$ by $\lambda$ and every other instance of $\lambda$ by $0$ in the above expressions.

\subsection{The Schr\"odinger representation}\label{ssSchrod}
Let $x_1,x_2$ and $x_3,x_4$ be the even resp. odd representatives of the coordinates on $\mc{P} (\mathds{R}^{2 |2}) / \mc{I}_\lambda$. Note that the operators occurring in $\pi_\lambda$ are not only well-defined on polynomials, but can be extended to smooth functions. For example $\pil (X) \exp(-2(x_1+x_2))$ is well-defined for all $X\in \mg$. Furthermore note that $\exp(-2(x_1+x_2))$ is invariant under the action of $\mf{k}$. This allows us to define the Schr\"odinger representation as follows:
\begin{align*}
W_\lambda &:= \mc{P}(\mathds{R}^{2|2}  )\exp(-2(x_1+x_2)) \mod \Ila,
\end{align*}
where the $\mf g$ module structure is given by $\pil$.

We will see in Section \ref{Segal-Bargmann transform} that we have an explicit $\mf{k}$-finite decomposition of $W_\lambda$.

\section{The Fock space and Bessel-Fischer product}
\label{Section Fock space}
In this section we will investigate the space $\mc{P} (\C^{2|2}) / \Ila$, for $\lambda \in \{1,\alpha\}$, on which the Fock representation is defined. In particular we will show that we have an non-degenerate, superhermitian, sesquilinear form on it. 
In \cite[Section 2.3]{HKMO} an inner product on the polynomial space $\mc P(\C^m)$ was introduced, namely the Bessel-Fischer inner product 
\begin{align*}
\bfip{p,q} := \left. p(\bessel)\bar q(z)\right|_{z=0}.
\end{align*}
Here $p(\bessel)$ is obtained by replacing $z_i$ by $\bessel (z_i)$ and $\bar q(z) = \overline{q(\bar z)}$ is obtained by conjugating the coefficients of the polynomial $q$. For polynomials, it was proven that the Bessel-Fischer inner product is equal to the $L^2$-inner product of the Fock space \cite[Proposition 2.6]{HKMO}. In \cite{BCD} this product was used as the starting point to generalize the Fock space to superspace. The Bessel-Fischer product is no longer an inner product in this case. However, under the right conditions it is a non-degenerate, sesquilinear, superhermitian form, which is consistent with the definition of Hilbert superspaces given in \cite{dGM}. We will use the same approach in this section to construct the Fock space associated with $D(2,1;\alpha)$.

A straightforward verification shows that the Bessel-Fischer product constructed in this section will be degenerate for $\alpha=0$. For $\alpha =1$  we know from \cite[Proposition 4.13]{BCD} that  the Bessel-Fischer product is degenerate. Therefore, from now on we will always assume $\lambda\in\{1,\alpha\}$ with $\alpha\not=0$ and $\alpha\not=1$.

\subsection{The Bessel-Fischer product}

\begin{Def}
For $p, q\in \mathcal P(\C^{2|2})$ we define the \textbf{Bessel-Fischer product} of $p$ and $q$ as
\begin{align*}
\bfip{p,q} := \left. p(\bessel)\bar q(z)\right|_{z=0},
\end{align*}
where $\bar q(z) = \overline{q(\bar z)}$ is obtained by conjugating the coefficients of the polynomial $q$ and $z :=(z_1, z_2, z_3, z_4)\in \C^{2|2}$.
\end{Def}

Explicitly for $p=\sum_\gamma a_\gamma z^\gamma$ and $q=\sum_\beta b_\beta z^\beta$ we have
\begin{align*}
\bfip{p,q} = \sum_{\gamma, \beta}\left.a_\gamma \bar b_\beta \bessel(z_1)^{\gamma_1}\bessel(z_2)^{\gamma_{2}}\bessel(z_3)^{\gamma_3}\bessel(z_4)^{\gamma_4}z_1^{\beta_1}z_2^{\beta_2}z_3^{\beta_3}z_4^{\beta_4}\right|_{z_1=z_2=z_3=z_4=0}.
\end{align*}

\begin{Def}\label{DefFock}
We define the \textbf{polynomial Fock space} as the superspace
\begin{align*}
\Fock := \mathcal{P}(\C^{2|2})/ \mc I_\lambda,
\end{align*}
with $\mc I_\lambda$ as in equation (\ref{Eq I lambda}).
\end{Def}

Let $z_1, z_2$ be the even coordinates and $z_3, z_4$ the odd coordinates of $\mathcal{P}(\C^{2|2})$. Then, for $\lambda=\alpha$, a homogeneous element $p$ of degree $k\geq 1$ in $\mathcal{F}_\alpha$ can be represented by 
\begin{align*}
p \equiv z_1^{k-1}\sum_{i=1}^4 p_i z_i, 
\end{align*}
with $p_i$ in $\C$.
Similarly, for $\lambda=1$, a homogeneous element $p$ of degree $k\geq 1$ in $\mathcal{F}_1$ can be represented by 
\begin{align*}
p \equiv z_2^{k-1}\sum_{i=1}^4 p_i z_i.
\end{align*}

We now prove that the Bessel-Fischer product is a non-degenerate, sesquilinear, superhermitian form when restricted to $\Fock$. The sesquilinearity follows directly from the linearity of the Bessel operators. Note that for all $p, q\in \mc P(\C^{2|2})$ we have
\begin{align*}
\bfip{z_i p,q} = (-1)^{|i||p|}\bfip{ p,\bessel(z_i)q}.
\end{align*}
This is a direct consequence of the definition of the Bessel-Fischer product.

\begin{Prop}[Orthogonality]\label{PropOrthog}
For $p_k\in \mathcal P_k(\C^{2|2})$ and $p_l\in \mathcal P_l(\C^{2|2})$ with $l\neq k$ we have
$\bfip{p_k, p_l} = 0.$
\end{Prop}
\begin{proof}
This follows from the fact that Bessel operators lower the degree of polynomials by one.
\qedhere
\end{proof}

\begin{lemma}\label{LemBess}
For $\lambda\in\{1,\alpha\}$ we have
\begin{align*}
\bessel(z_1)z_1 &= -\lambda, & \bessel(z_2)z_2 &= -\dfrac{\lambda}{\alpha}, & \bessel(z_3)z_4 &= -2\lambda, & \bessel(z_4)z_3 &= 2\lambda.
\end{align*}
and for all other cases we have $\bessel(z_i)z_j=0$.
\end{lemma}

\begin{proof}
We have
\begin{align*}
\bessel(z_1) &= (-\lambda+z_1\pt {z_1} + z_3\pt {z_3} + z_4\pt {z_4})\pt {z_1} -2\alpha {z_2} \pt {z_3} \pt {z_4}, \\
\bessel(z_2) &= (-\dfrac{\lambda}{\alpha}+{z_2}\pt {z_2} + z_3\pt {z_3} + z_4\pt {z_4})\pt {z_2} -2 z_1 \pt {z_3} \pt {z_4},\\
\bessel(z_3) &= (-2\lambda + 2z_1\pt {z_1} + 2\alpha {z_2}\pt {z_2} + 2(1+\alpha)z_3\pt {z_3})\pt {z_4} + {z_3} \pt {z_1} \pt {z_2},\\
\bessel(z_4) &= (2\lambda - 2z_1\pt {z_1} - 2\alpha {z_2}\pt {z_2} - 2(1+\alpha)z_4\pt {z_4})\pt {z_3} + z_4 \pt {z_1} \pt {z_2}.
\end{align*}
The lemma now follows from a straightforward verification.
\end{proof}

\begin{Prop}\label{Restriction}
For all $p\in \mc P(\C^{2|2})$ and $C \in \Ila$, with $\lambda\in\{1,\alpha\}$ it holds that
\begin{align*}
\bfip{C,p} =0 = \bfip{p,C}.
\end{align*}
Thus the Bessel-Fischer product is well defined on $\Fock$.  
\end{Prop}

\begin{proof}
Since $\Ila$ is a $\mg$-submodule, the Bessel operators map $\Ila$ to $\Ila$. Hence for an arbitrary $p\in \mathcal P(\C^{2|2})$ and $C \in \Ila$ there exists a $C' \in \Ila$ such that
\begin{align*}
\bfip{p, C} = p(\bessel)C(z) |_{z=0} = C'(z) |_{z=0}.
\end{align*}
Since an element in $\Ila$ has no constant term, this implies $\bfip{p,C} =0$. To prove $\bfip{C,p} =0$ we can restrict to $p \in\{z_i z_1, z_iz_2,z_i z_3, z_i z_4\}$  and $C \in \{2\frac{\alpha}{\lambda}z_1z_2+z_3z_4, z_j^2, z_j z_3, z_j z_4\}$ with $(z_i,z_j)=(z_1,z_2)$ for $\lambda=\alpha$ and $(z_i,z_j)=(z_2,z_1)$ for $\lambda=1$. The proposition now follows from verifying $\bessel(C)p = 0$ for all 32 cases. This is a straightforward calculation. For example
\begin{align*}
\bessel(2\frac{\alpha}{\lambda}z_1z_2+z_3z_4) z_1 z_2 &= 2\frac{\alpha}{\lambda} \bessel(z_1)\bessel(z_2)z_1 z_2 + \bessel(z_3)\bessel(z_4)z_1 z_2 \\
&= -2\bessel(z_1)z_1 +  \bessel(z_3)z_4\\
&= 2\lambda-2\lambda= 0,
\end{align*}
proving the case $C=2\frac{\alpha}{\lambda}z_1z_2+z_3z_4$ and $p=z_1z_2$.
\end{proof}

Since $\Fock$ is a relatively small superspace, we can calculate the Bessel-Fischer product explicitly.

\begin{Prop}\label{PropInnProd}
Suppose $\lambda=\alpha$ and $p,q\in \{z_1^{k}, z_1^kz_2,z_1^kz_3, z_1^kz_4\}$, with $k\in\N$. Then the only non-zero evaluations of $\bfip{p,q}$ are
\begin{align*}
\bfip{z_1^{k}, z_1^{k}} &= -\bfip{z_1^{k}z_2, z_1^{k}z_2} = k!(-\alpha)_{k},\\
\bfip{z_1^{k}z_3, z_1^{k}z_4} &= -\bfip{z_1^{k}z_4, z_1^{k}z_3}  = 2k!(-\alpha)_{k+1},
\end{align*}
where we used the Pochhammer symbol $(a)_k = a(a+1)(a+2)\cdots (a+k-1)$.\\
Similarly, suppose $\lambda=1$ and $p,q\in \{z_2^{k}, z_2^kz_1,z_2^kz_3, z_2^kz_4\}$, with $k\in\N$. Then the only non-zero evaluations of $\bfip{p,q}$ are
\begin{align*}
\bfip{z_2^{k}, z_2^{k}} &= -\bfip{z_2^{k}z_1, z_2^{k}z_1} = k!(-\alpha^{-1})_{k},\\
\bfip{z_2^{k}z_3, z_2^{k}z_4} &= -\bfip{z_2^{k}z_4, z_2^{k}z_3}  = -2 \alpha k!(-\alpha^{-1})_{k+1}.
\end{align*}
\end{Prop}

\begin{proof}
We prove the $\lambda=\alpha$ case. The $\lambda =1$ is entirely analogous. For the first non-zero case we find
\begin{align*}
\bfip{z_1^{k}, z_1^{k}} = \bfip{z_1^{k-1}, \bessel(z_1)z_1^{k}} = k(k-1-\alpha)\bfip{z_1^{k-1}, z_1^{k-1}}.
\end{align*}
Iterating this process we obtain
\begin{align*}
\bfip{z_1^{k}, z_1^{k}} = (-1)^{k}k!(\alpha-k+1)_{k}\bfip{1,1} = k!(-\alpha)_{k}.
\end{align*}
The other non-zero cases now follow from
\begin{align*}
\bfip{z_1^{k}z_2, z_1^{k}z_2} &= \bfip{z_1^k, \bessel(z_2)z_1^{k}z_2} = -\bfip{z_1^k, z_1^k},\\
\bfip{z_1^{k}z_3, z_1^{k}z_4} &= \bfip{z_1^k, \bessel(z_3)z_1^{k}z_4} = 2(k-\alpha)\bfip{z_1^k, z_1^k},\\
\bfip{z_1^{k}z_4, z_1^{k}z_3} &= \bfip{z_1^k, \bessel(z_4)z_1^{k}z_3} = -2(k-\alpha)\bfip{z_1^k, z_1^k}.
\end{align*}
The general case follows from similar calculations and iterations, taking into account Lemma \ref{LemBess}.
\end{proof}

This property also shows us that the Bessel-Fischer product is superhermitian and tells us when it is non-degenerate.

\begin{Cor}[Superhermitianity]
For $\lambda\in\{1,\alpha\}$ the Bessel-Fischer product is superhermitian on $\Fock$, i.e.,
\begin{align*}
\bfip{p, q} = (-1)^{\left|p\right|\left|q\right|}\bfipbar{q, p},
\end{align*}
for $p, q\in \Fock$.
\end{Cor}

\begin{Cor}[Non-degenaracy]\label{nondeg}
For $\lambda=\alpha$ with $\alpha\not\in \N$ and for $\lambda=1$ with $\alpha^{-1}\not\in \N$ the Bessel-Fischer product is non-degenerate on $\Fock$, i.e., if
$
\bfip{p,q} =0,
$
for all $q\in \Fock$, then $p=0$.
\end{Cor}

\subsection{Reproducing kernel}\label{SSRepKer}

In the classical case a reproducing kernel for the Fock space was constructed in Section 2.4 of \cite{HKMO}. A generalisation of this reproducing kernel in superspace was constructed in \cite{BCD} for $\mf{osp}(m,2|2n)$. Similarly, we can construct a ``reproducing kernel'' for $D(2,1;\alpha)$.

Suppose $w := (w_1,w_2,w_3,w_4) \in \K^{2|2}$. We define
\begin{align*}
z|w := z_1 w_1+\alpha z_2 w_2 - \dfrac{1}{2}z_3 w_4 +\dfrac{1}{2}z_4 w_3.
\end{align*}

\begin{lemma}\label{Lemma rep kernel}
Define the superfunction $\ds I_{\lambda,k}(z,w)$, with $\lambda\in\{1,\alpha\}$, by
\begin{align*}
\ds I_{\alpha,k}(z,w) &:= \dfrac{(-1)^k}{k!}\left(\alpha-k+1\right)_k^{-1}(z|\overline{w})^k,\\
\ds I_{1,k}(z,w) &:= \dfrac{(-1)^k}{\alpha^{k}k!}\left(\alpha^{-1}-k+1\right)_k^{-1}(z|\overline{w})^k,
\end{align*}
where we used the Pochhammer symbol $(a)_k = a(a+1)(a+2)\cdots (a+k-1)$. For all $p\in \mc P_k(\C^{2|2})$ we have
\begin{align*}
\bfip{p(z),\ds I_{\lambda,k}(z,w)} = p(w) \mod  \mc I_\lambda.
\end{align*}
\end{lemma}

\begin{proof}
First note that
\begin{align*}
\bessel(z_i)(z|w) &= -\lambda w_i = -\dfrac{\lambda}{\alpha}\alpha w_i.
\end{align*}
For $\lambda=\alpha$ we have
\begin{align*}
\bessel(z_1)(z|w)^k &= (-\alpha+\E)\pt {z_1} (z|w)^k - z_2(\pt {z_1} \pt {z_2} +2\alpha\pt {z_3}\pt {z_4})(z|w)^k\\
&= -k(\alpha-k+1)w_1(z|w)^{k-1} - z_2k(k-1)(\alpha w_1w_2+\dfrac{1}{2}\alpha w_3 w_4)(z|w)^{k-2}\\
&= -k(\alpha-k+1)w_1(z|w)^{k-1} - \Rsq(w)\dfrac{\alpha}{2} z_2k(k-1) (z|w)^{k-2},
\end{align*}
where $\E := z_1\pt{z_1}+z_2\pt{z_2}+ z_3\pt{z_3}+z_4\pt{z_4}$ is the Euler operator. Iterating the previous calculation and working modulo $\mc I_\alpha$ we obtain
\begin{align*}
\bessel(z_1)^a(z|w)^k &= (-1)^a(k(k-1)\cdots(k-a))(\alpha-k+1)_a w_1^a(z|w)^{k-a},
\end{align*}
for $a\leq k$. Now suppose $p$ is a monomial. Since the Bessel-Fischer product restricts to $\Fock$ we may assume $p(z) = z_1^{k-1}z_i$. We have
\begin{align*}
\bfip{p(z),(z|\overline{w})^k} &= \bfip{z_i,\bessel(z_1)^{k-1}(z|\overline{w})^k}\\
 &= (-1)^{k-1}k!(\alpha-k+1)_{k-1} w_1^{k-1}\bfip{z_i,(z|\overline{w})}\\
 &= (-1)^{k}k!(\alpha-k+1)_{k} w_1^{k-1}w_i\\
 &=(-1)^{k}k!(\alpha-k+1)_{k}p(w),
\end{align*}
which gives us the desired result. The case $\lambda=1$ is entirely analogous.
\qedhere
\end{proof}

We will give a closed formula of the reproducing kernel in terms of the renormalised I-Bessel function. The I-Bessel function $I_\gamma(t)$ (or modified Bessel function of the first kind) is defined by
\begin{align*}
I_\gamma(t) := \left(\dfrac{t}{2}\right)^{\gamma}\sum_{k=0}^\infty \dfrac{1}{k!\Gamma(k+\gamma+1)}\left(\dfrac{t}{2}\right)^{2k},
\end{align*}
for $\gamma, t\in\C$, see \cite{AAR}, Section 4.12. We will use the renormalisation
\begin{align*}
\widetilde I_\gamma(t) := \left(\dfrac{t}{2}\right)^{-\gamma} I_\gamma(t).
\end{align*}

\begin{theorem}[Reproducing kernel of $\Fock$]\label{Theorem repr kernel}
Suppose either $\lambda=\alpha$ with $\alpha\not\in \N$ or $\lambda=1$ with $\alpha^{-1}\not\in \N$. Define the superfunction $\ds I_\lambda(z,w)$ by
\begin{align*}
\ds I_\alpha(z,w) &:= \Gamma(-\alpha)\widetilde I_{-1-\alpha}\left(2\sqrt{(z|\overline w)}\right),\\
\ds I_1(z,w) &:= \Gamma(-\alpha^{-1})\widetilde I_{-1-\alpha^{-1}}\left(2\sqrt{\alpha^{-1}(z|\overline w)}\right).
\end{align*}
For all $p\in \Fock$ we have
\begin{align*}
\bfip{p(z),\ds I_\lambda(z,w)} = p(w).
\end{align*}
\end{theorem}

\begin{proof}
Note that 
\begin{align*}
 \Gamma(-\alpha)\widetilde I_{-1-\alpha}\left(2\sqrt{(z|\overline w)}\right) &= \sum_{k=0}^\infty \dfrac{1}{k!}\dfrac{\Gamma(-\alpha)}{\Gamma(k-\alpha)}(z|\overline{w})^k  \\
&=  \sum_{k=0}^\infty \dfrac{(-1)^k}{k!}\left(\alpha-k+1\right)_k^{-1}(z|\overline{w})^k \\
&=\sum_{k=0}^\infty \ds I_{\alpha,k}(z,w), 
\end{align*}
and similarly
\begin{align*}
 \Gamma(-\alpha^{-1})\widetilde I_{-1-\alpha^{-1}}\left(2\sqrt{\alpha^{-1}(z|\overline w)}\right)=\sum_{k=0}^\infty \ds I_{1,k}(z,w) .
\end{align*}
The proposition then follows from Lemma \ref{Lemma rep kernel} and the orthogonality property.
\end{proof}

\section{Properties of the Fock Representation}
We will now show that the Fock representation is skew-symmetric with respect to the Bessel-Fischer product. Furthermore, we will also give an explicit $\mf{k}$- and $\mf{k}_0$-decomposition of the Fock model. Just like in Section \ref{Section Fock space}, we will again assume $\lambda \in \{1,\alpha\}$ and $\alpha\not=0$ and $\alpha\not=1$.

\subsection{Skew-symmetric}
We start with some preparatory lemma's. 
\begin{lemma}\label{LemEulerish}
Let $p,q\in \Fock$ be two homogeneous elements of degree $k\geq 1$ and define $(z^1, z^2, z^3, z^4) := (z_1, z_2, z_4, z_3) $. For $\lambda=\alpha$ and $j\neq 1$ we have
\begin{align*}
\bfip{z_i \pt {z_j} p, q} &= p_j\overline{q}_{z^i}\bfip{z_1^{k-1}z_i,z_1^{k-1}z^i},\\
\bfip{{z_1} \pt {{z_1}} p, q} &= (k-1)\bfip{p,q}+p_1\overline{q}_{1}\bfip{z_1^k,z_1^k}, \\
\bfip{{z_2} \pt {{z_1}} p, q} &= kp_1\ol q_2\bfip{z_1^{k-1}{z_2}, z_1^{k-1}{z_2}}, \\
\bfip{z_3 \pt {{z_1}} p, q} &=k p_1\ol q_4\bfip{z_1^{k-1}z_3, z_1^{k-1}{z_4}}-2(k-1)p_4\ol q_2\bfip{z_1^{k-1}{z_2}, z_1^{k-1}{z_2}}, \\
\bfip{{z_4} \pt {{z_1}} p, q} &=-k p_1\ol q_3\bfip{z_1^{k-1}z_3, z_1^{k-1}{z_4}}+2(k-1)p_3\ol q_2\bfip{z_1^{k-1}{z_2}, z_1^{k-1}{z_2}},
\end{align*}
where $q_{z^i}$ denotes the coefficient of the $z_1^{k-1}z^i$ term in $q$.
\end{lemma}

\begin{proof}
If we use Proposition \ref{PropInnProd} and the fact that we are working modulo $\Ial$, we obtain
\begin{align*}
\bfip{z_i \pt {z_j} p, q} &= \bfip{z_i \pt {z_j} (p_1 {z_1}+p_2 {z_2}+p_3 z_3+p_4 {z_4})z_1^{k-1}, q} = p_{z_j}\bfip{z_iz_1^{k-1}, q}\\
&= p_{z_j}\overline{q}_{z^i}\bfip{z_1^{k-1}z_i,z_1^{k-1}z^i},\\
\bfip{z_1\pt {z_1} p, q} &= \bfip{{z_1} (k p_1 {z_1}+ (k-1)(p_2 {z_2}+p_3 z_3+p_4{z_4}))z_1^{k-2}, q}\\
&= (k-1)\bfip{p,q}+p_1\overline{q}_{z_1}\bfip{z_1^k,z_1^k},\\
\bfip{{z_2}\pt {z_1} p, q} &= \bfip{{z_2} (k p_1 {z_1}+ (k-1)(p_2 {z_2}+p_3 z_3+p_4{z_4}))z_1^{k-2}, q}\\
&= kp_1\ol q_2\bfip{z_1^{k-1}{z_2}, z_1^{k-1}{z_2}},\\
\bfip{z_3\pt {z_1} p, q} &= \bfip{z_3 (k p_1 {z_1}+ (k-1)(p_2 {z_2}+p_3 z_3+p_4{z_4}))z_1^{k-2}, q}\\
&= \bfip{k p_1 z_1^{k-1}z_3 + (k-1)p_4 z_1^{k-2}z_3{z_4}, q}\\
&= k p_1 \bfip{z_1^{k-1}z_3, q} + (k-1)p_4\bfip{z_1^{k-2}(-2{z_1}{z_2}), q}\\
&= k p_1\ol q_4\bfip{z_1^{k-1}z_3, z_1^{k-1}{z_4}}-2(k-1)p_4\ol q_2\bfip{z_1^{k-1}{z_2}, z_1^{k-1}{z_2}}\\
\end{align*}
and
\begin{align*}
\bfip{{z_4}\pt {z_1} p, q} &= \bfip{{z_4} (k p_1 {z_1}+ (k-1)(p_2 {z_2}+p_3 z_3+p_4{z_4}))z_1^{k-2}, q}\\
&= \bfip{k p_1 z_1^{k-1}{z_4} + (k-1)p_3 z_1^{k-2}{z_4} z_3, q}\\
&= k p_1 \bfip{z_1^{k-1}{z_4}, q} + (k-1)p_3\bfip{z_1^{k-2}(2{z_1}{z_2}), q}\\
&= -k p_1\ol q_3\bfip{z_1^{k-1}z_3, z_1^{k-1}{z_4}} + 2(k-1)p_3\ol q_2\bfip{z_1^{k-1}{z_2}, z_1^{k-1}{z_2}},
\end{align*}
proving the lemma.
\end{proof}

\begin{lemma}\label{LemSkewSymRho}
For all $p,q\in \Fock$, with $\lambda\in\{1,\alpha\}$, we have
\begin{align*}
\bfip{{z_1} \pt {z_1} p, q} &= \bfip{p, {z_1} \pt {z_1} q}, & \bfip{{z_2} \pt {z_2} p, q} &= \bfip{p, {z_2} \pt {z_2} q},\\
\bfip{z_3 \pt{z_4} p, q} &= -\bfip{p,{z_3} \pt {z_4} q}, & \bfip{{z_4} \pt{ z_3} p, q} &= -\bfip{p,{z_4} \pt {z_3} q},\\
\bfip{z_3 \pt {z_3} p, q} &= \bfip{p,{z_4} \pt{z_4} q}, & \bfip{{z_4} \pt{z_4} p, q} &= \bfip{p, z_3 \pt {z_3} q}.
\end{align*}
\end{lemma}

\begin{proof}
This follows directly from Lemma \ref{LemEulerish} for $\lambda=\alpha$ and from an entirely analogous version of Lemma \ref{LemEulerish} for $\lambda=1$.
\end{proof}

\begin{Prop}\label{PropSkewSymRho}
The Fock representation $\rol$ on $\Fock$, with $\lambda\in\{1,\alpha\}$, is skew-supersymmetric with respect to the Bessel-Fischer product, i.e.,
\begin{align*}
\bfip{\rol(X)p,q} = - (-1)^{|X||p|}\bfip{p,\rol(X)q},
\end{align*}
for all $X\in \mathfrak{g}$ and $p,q\in \Fock$.
\end{Prop}

\begin{proof}
Once again we only consider the $\lambda=\alpha$ case, the $\lambda=1$ being entirely analogous. Suppose $p,q\in \Fock$ are homogeneous polynomials of degree $k$ and $a\in D_\alpha$. We have
\begin{align*}
\bfip{\rol(a,0,a)p,q} &= \bfip{-\imath(a+\bessel(a)) p,q} =  -\imath(-1)^{|a||p|}\bfip{ p,(\bessel(a)+a)q}\\&=- (-1)^{|a||p|}\bfip{p,\rol(a,0,a)q}.
\end{align*}
Now we look at
\begin{align*}
\bfip{\rol(a,0,-a)p,q} = \bfip{\pil(2\imath L_{a})p,q}.
\end{align*}
For $a\in\{e_1,e_2\}$ we use Lemma \ref{LemSkewSymRho} to see
\begin{align*}
\bfip{\pil(2\imath L_{a})p,q} = -\bfip{p,\pil(2\imath L_{a})q}.
\end{align*}
For $a=\xi$ we need to prove
\begin{align*}
\bfip{\pil(2\imath L_{\xi})p,q} = -(-1)^{|p|}\bfip{p,\pil(2\imath L_{\xi})q},
\end{align*}
or, equivalently,
\begin{align*}
\bfip{\pil(2 L_{\xi})p,q} = (-1)^{|p|}\bfip{p,\pil(2 L_{\xi})q}.
\end{align*}
Using Lemma \ref{LemEulerish} we find
\begin{align*}
\bfip{\pil(2 L_{\xi})p,q} &=  \bfip{-{z_3}(\pt {z_1}+\pt {z_2} )p, q} + \bfip{-2({z_1}+\alpha z_2)\pt{z_4} p, q} \\
&= +k p_1 \ol q_4\bfip{z_1^{k-1}{z_3}, z_1^{k-1}{z_4}} - 2(k-1)p_4 \ol q_2 \bfip{z_1^{k-1}{z_2}, z_1^{k-1}{z_2}}\\
&\quad  +2\alpha p_4 \ol q_2 \bfip{z_1^{k-1}{z_2}, z_1^{k-1}{z_2}}+ 2 p_4\ol q_1 \bfip{z_1^{k}, z_1^{k}}\\
&\quad + p_2 \ol q_4\bfip{z_1^{k-1}{z_3}, z_1^{k-1}{z_4}}\\
&=  2(\alpha-k+1)p_4\ol q_2 \bfip{z_1^{k-1}{z_2}, z_1^{k-1}{z_2}}\\
&\quad +(k p_1 +p_2)\ol q_4 \bfip{z_1^{k-1}{z_3}, z_1^{k-1}{z_4}} +2p_4\ol q_1 \bfip{z_1^{k}, z_1^{k}}
\end{align*}
and similarly
\begin{align*}
-\bfip{p,\pil(2L_{\xi})q} &= \bfip{ p, {z_3}\pt {z_1}q} +2\alpha\bfip{ p,{z_2}\pt{z_4} q} + \bfip{p,{z_3}\pt {z_2} q}+2\bfip{ p, {z_1}\pt{z_4} q}\\
&=  2(\alpha-k+1)p_2\ol q_4 \bfip{z_1^{k-1}{z_2}, z_1^{k-1}{z_2}}\\
&\quad -(k\ol q_1  +\ol q_2)p_4 \bfip{z_1^{k-1}{z_3}, z_1^{k-1}{z_4}} +2p_1\ol q_4 \bfip{z_1^{k}, z_1^{k}}
\end{align*}
If $|p|=0$, then $p_3 = p_4 = 0$ and what we need to prove reduces to
\begin{align*}
(k p_1 +p_2) \bfip{z_1^{k-1}{z_3}, z_1^{k-1}{z_4}} &=  2(\alpha-k+1)p_2 \bfip{z_1^{k-1}{z_2}, z_1^{k-1}{z_2}} \\
&\quad +2p_1\bfip{z_1^{k}, z_1^{k}}.
\end{align*}
Separating the $p_1$ and $p_2$ terms we obtain
\begin{align*}
\left\lbrace\begin{array}{lcl}
k \bfip{z_1^{k-1}{z_3}, z_1^{k-1}{z_4}} &=&  2\bfip{z_1^{k}, z_1^{k}},\\
\bfip{z_1^{k-1}{z_3}, z_1^{k-1}{z_4}} &=&  2(\alpha-k+1)p_2 \bfip{z_1^{k-1}{z_2}, z_1^{k-1}{z_2}},
\end{array}\right.
\end{align*}
which holds because of Proposition \ref{PropInnProd}.
If $|p|=1$, then $p_1 = p_2 = 0$ and what we need to prove reduces to
\begin{align*}
(k\ol q_1  +\ol q_2) \bfip{z_1^{k-1}{z_3}, z_1^{k-1}{z_4}} &= 2(\alpha-k+1)\ol q_2 \bfip{z_1^{k-1}{z_2}, z_1^{k-1}{z_2}}\\
&\quad +2\ol q_1 \bfip{z_1^{k}, z_1^{k}}
\end{align*}
which holds similarly. The case $a=\eta$ is entirely analogous. Similar calculations also give us
\begin{align*}
\bfip{\rol(0, 4[L_{e_1}, L_\xi], 0))p,q} = - (-1)^{|p|}\bfip{p,\rol(0, 4[L_{e_1}, L_{\xi}], 0))q},\\
\bfip{\rol(0, 4[L_{e_1}, L_{\eta}], 0))p,q} = - (-1)^{|p|}\bfip{p,\rol(0, 4[L_{e_1}, L_{\eta}], 0))q}.
\end{align*}
The remaining cases, which are $X = (0, 4[L_{\xi}, L_\xi], 0)$, $X=(0, 4[L_{\eta}, L_{\eta}], 0)$ and $X=(0, 4[L_{\xi}, L_\eta], 0)$ for $\alpha \neq -1$ and $X = (0, d_-, 0)$, $X=(0, d_+, 0)$ and $X=(0, d_0, 0)$ for $\alpha = -1$, follow directly from Lemma \ref{LemSkewSymRho}.
\qedhere
\end{proof}

\subsection{The $(\mf g, \mf k)$-module $F_\lambda$}

We define
\begin{align*}
F_\lambda &:= U(\mathfrak{g})1 \mod \Ila,
\end{align*}
where the $\mf g$-module structure is given by the Fock representation $\rol$. For $k\geq 1$ we introduce
\begin{align*}
F_{\lambda,0} &:= \C,\\
F_{\lambda,k} &:= (\C {z_1} \oplus \C {z_2} \oplus \C {z_3} \oplus \C {z_4})z_i^{k-1},\\
H_{\alpha, k} &:= (\C ({z_1} + (k-1-\alpha){z_2} )\oplus \C {z_3} \oplus \C {z_4})z_1^{k-1},\\
H_{1, k} &:= (\C ({z_2} +(k-1-\alpha^{-1}){z_1} )\oplus \C {z_3} \oplus \C {z_4})z_2^{k-1},
\end{align*}
where $i=1$ if $\lambda=\alpha$ and $i=2$ if $\lambda=1$.

\begin{theorem}[Decomposition of $F$]\label{ThDecF}
Suppose either $\lambda=\alpha$ and $\alpha\not\in \N$ or $\lambda=1$ and $\alpha^{-1}\not\in \N$. We have the following:
\begin{itemize}
\item[(1)] \begin{itemize}
\item[(a)] For $\alpha \neq -1$ an explicit decomposition of $F_{\lambda,k}$ into irreducible $\mathfrak{k}_{0}$-modules is given by
\begin{align*}
F_{\lambda,k} = H_{\lambda,k} \oplus \C ({z_1} + {z_2})^k.
\end{align*}
\item[(b)] If $\alpha = -1$, then $F_{\lambda,k}$ is an indecomposable $\mathfrak{k}_{0}$-module, but not an irreducible $\mathfrak{k}_{0}$-module.
\end{itemize}
\item[(2)] $F_{\lambda,k}$ is an irreducible $\mf k$-module.
\item[(3)] $F_\lambda$ is an irreducible $\mathfrak{g}$-module and its $\mf k$-type decomposition is given by
\begin{align*}
F_{\lambda} = \bigoplus_{k=0}^\infty F_{\lambda,k}.
\end{align*}
\end{itemize}
\end{theorem}

\begin{proof}
We give the proof for $\lambda=\alpha$. The $\lambda=1$ case is similar. Note that because we are working modulo $\Ial$, we have $({z_1} + {z_2})^k = z_1^{k-1}({z_1} + k{z_2})$. First suppose $\alpha\neq -1$. The action of $\mathfrak{k}_{0}$ on $F_{\lambda,k}$ is given by the following table:
\begin{align*}
\def\arraystretch{2}
\begin{array}{l||l|l|l|l}
 & z_1^k & z_1^{k-1}{z_2} & z_1^{k-1}{z_3} & z_1^{k-1}{z_4}\\ \hline\hline
4[L_{e_1},L_\xi]  & -k z_1^{k-1}{z_3} & z_1^{k-1}{z_3} & 0 & 2z_1^k \\
& & & & -2(\alpha-k+1)z_1^{k-1}{z_2}\\ \hline
4[L_{e_1},L_\eta]  & -k z_1^{k-1}{z_4} & z_1^{k-1}{z_4} & 2(\alpha-k+1)z_1^{k-1}{z_2} & 0\\
& & & -2z_1^k &\\ \hline
4[L_{{\xi}},L_{\xi}]  & 0 & 0 & 0 & 4(1+\alpha)z_1^{k-1}{z_3}\\ \hline
4[L_{{\eta}},L_{\eta}]  & 0 & 0 & -4(1+\alpha)z_1^{k-1}{z_4} & 0\\ \hline
4[L_{{\xi}},L_{\eta}]  & 0 & 0 & -2(1+\alpha)z_1^{k-1}{z_3} & 2(1+\alpha)z_1^{k-1}{z_4}
\end{array}
\end{align*}
This gives us (1)(a) after a straightforward verification. To prove (2) we now only need an element mapping $z_1^{k-1}{z_2}$ to $({z_1}+k{z_2})z_1^{k-1}$ and the reverse for every $k\in\N$. We find that
\begin{align*}
\rho_k &:= \rol(2[L_{e_1},L_\xi])\circ\rol(4[L_{e_1},L_\eta])-(1+\alpha)(\rol(f_2)- \rol(e_2))^2,\\
\rho'_k &:= \dfrac{(\rol(f_2)- \rol(e_2))}{2(\alpha-2k+1)}\circ((2k-\alpha)(\rol(f_2)- \rol(e_2))+ (\rol(f_1)- \rol(e_1)))
\end{align*}
are such elements. Define the following two elements of the action of $\mf g$:
\begin{align*}
\rho^+ &:= \rol(c^{-1}(\frac{-e_1}{2},0,0)) =\imath {z_1},\\
\rho^- &:= \rol(c^{-1}(0,0,-2e_1)) = \imath\bessel({z_1}).
\end{align*}
We find
\begin{align*}
\rho^+(z_1^k) &= \imath z_1^{k+1},\\
\rho^-(z_1^k) &= \imath\bessel({z_1}) z_1^k= \imath k(-\alpha+k-1)z_1^{k-1}.
\end{align*}
which shows that $\rho^+$ allows us to go to polynomials of higher degrees while $\rho^-$ allows us to go the other direction for $\alpha\not\in \N$. Therefore we obtain $(3)$. The $\alpha =-1$ case is entirely analogous, except that now $2[L_{e_1},L_{\xi}]\in \mf k_0$ sends $z_1^{k-1}z_4 \in H_{\alpha, k}$ to $(z_1+z_2)^k$. This implies that the decomposition of $F_{\lambda,k}$ no longer holds and since $\C ({z_1} + {z_2})^k$ is still an irreducible component in $F_{\lambda,k}$, this proves (1)(b).
\end{proof}

The following isomorphism is an immediate result of this theorem.

\begin{Cor}
Suppose either $\lambda=\alpha$ and $\alpha\not\in \N$ or $\lambda=1$ and $\alpha^{-1}\not\in \N$. Let
\begin{align*}
\Fock = \mathcal{P}(\C^{2|2})/{\mc I_\lambda}
\end{align*}
be the polynomial Fock space defined in Definition \ref{DefFock}. We have $\Fock \cong F_\lambda$.
\end{Cor}

Recall the $\mf{sl}(2)$-triple $\{e,f,h\}$ with $e=e_1+e_2$, $f=f_1+f_2$ and $h=2L_e$ from Section \ref{Three grading}. Using the Cayley transform $c$ we obtain another $\mf{sl}(2)$-triple $\mf s :=\{\tilde e,\tilde f,\tilde h\}$ where 
\begin{align*}
\tilde e := c^{-1}(e), \quad \tilde f &:= c^{-1}(f) \quad \mbox{ and } \quad \tilde h := c^{-1}(h).
\end{align*}
We have
\begin{align*}
\rol(\tilde f) &= \pil(f) = -2\imath ({z_1}+{z_2}),\\
\rol(\tilde h) &= \pil(h) = \lambda(1+\alpha^{-1})-2\E,\\
\rol(\tilde e) &= \pil(e) = -\dfrac{\imath}{2}(\bessel({z_1})+\bessel({z_2})),
\end{align*}
where $\E := {z_1}\pt {z_1} + {z_2}\pt {z_2} + {z_3}\pt{z_3} +{z_4} \pt {z_4}$ is the Euler operator. In particular we have
\begin{align*}
\rol(\tilde f)({z_1}+{z_2})^k &= -2\imath({z_1}+{z_2})^{k+1},\\
\rol(\tilde h)({z_1}+{z_2})^k &= (\lambda(1+\alpha^{-1})-2k)({z_1}+{z_2})^k,\\
\rol(\tilde e)({z_1}+{z_2})^k &= -\dfrac{\imath}{2}(k(-\lambda(1+\alpha^{-1})+k-1)({z_1}+{z_2})^{k-1},
\end{align*}
for $k\in \N$. For $\lambda=\alpha$ we also have
\begin{align*}
\rol(\tilde f)(({z_1} + (k-1-\alpha){z_2})z_1^{k-1}) &= -2\imath({z_1}+(k-\alpha){z_2})z_1^{k},\\
\rol(\tilde h)(({z_1} + (k-1-\alpha){z_2})z_1^{k-1}) &= (1+\alpha-2k)({z_1} + (k-1-\alpha){z_2})z_1^{k-1},\\
\rol(\tilde e)(({z_1} + (k-1-\alpha){z_2})z_1^{k-1}) &= -\dfrac{\imath}{2}(k-1)(k-1-\alpha)({z_1} + (k-2-\alpha){z_2})z_1^{k-2},
\end{align*}
where entirely analogous results hold for $\lambda=1$.
Define $G_\lambda := \sum_{k=0}^\infty \C ({z_1}+{z_2})^k$ and $H_\lambda := \sum_{k=0}^\infty H_{\lambda,k}$, then it is clear that $G_\lambda$ and $H_\lambda$ are invariant under the action of $\mf s$. Note also that $\mf s$ and $\mf k_0$ commute. This gives us the following theorem.

\begin{theorem}
Suppose either $\lambda=\alpha$ and $\alpha\not\in \N\cup \{-1\}$ or $\lambda=1$ and $\alpha^{-1}\not\in \N\cup \{-1\}$. Under the action of $(\mf s, \mf k_0)$ we have the decomposition
\begin{align*}
F_\lambda = G_\lambda \oplus H_\lambda.
\end{align*}
\end{theorem}

\subsection{The Gelfand-Kirillov dimension}

The Gelfand-Kirillov dimension is a measure of the size of a representation that roughly measures how fast a representation grows to infinity. In particular, the  Gelfand-Kirillov dimension is zero for finite-dimensional representations. Minimal representations have the property that they attain the lowest possible Gelfand-Kirillov dimension of all infinite-dimensional representations \cite{GanSavin}. Since we will show that the Gelfand-Kirillov dimension of our representation is $1$, the Fock model also has the lowest possible Gelfand-Kirillov dimension. 

 Let $R$ be a finitely generated algebra, then the Gelfand-Kirillov dimension of a finitely generated $R$-module $F$ is defined by
\begin{align*}
GK(F) = \lim\sup_{k\rightarrow \infty} \left(\log_k\dim(V^kF_0)\right).
\end{align*}
Here $V$ is a finite-dimensional subspace of $R$ which contains the unit element $1$ and generators of $R$, and $F_0$ is a finite-dimensional subspace of $F$, which generates $F$ as an $R$-module. The definition is independent of the chosen $V$ and $F_0$, see \cite[Section 7.3]{Mu}.

\begin{Prop} \label{Prop Gelfand-Kirillov dimension}
Suppose either $\lambda=\alpha$ and $\alpha\not\in \N$ or $\lambda=1$ and $\alpha^{-1}\not\in \N$. The Gelfand-Kirillov dimension of the $U(\mf g)$-module $F_\lambda$ is given by $GK(F_\lambda) = 1$.
\end{Prop}

\begin{proof}
We choose $F_{\lambda,0}$ for $F_0$ and $\mf g \oplus 1\subset U(\mf g)$ for $V$. Then $V^k = U_k(\mf g)$ is the canonical filtration on the universal enveloping algebra.
We have
\begin{align*}
\dim\left(U_k(\mf g) F_{\lambda,0}\right) &= \dim\left(\bigoplus_{j=0}^k F_{\lambda,j}\right) = \sum_{j=0}^k \dim(F_{\lambda,k}) = 1 + \sum_{j=1}^k 4 = 1+4k
\end{align*}
and therefore
\begin{align*}
GK(F_\lambda) = \lim\sup_{k\rightarrow \infty} \left(\log_k(1+4k)\right) = 1,
\end{align*}
which is what we wished to prove.
\end{proof}

\section{The Segal-Bargmann transform}
\label{Section SB transform}
In this section we construct the Segal-Bargmann transform as an operator that intertwines the Schrodinger representation and the Fock representation. We start by looking at an intertwining operator between both representations that occurs naturally when we restrict the Fock representation to $\R$. This intertwining operator can then be used in combination with the reproducing kernel from Section \ref{SSRepKer} to define the Segal-Bargmann transform. We then use the Segal-Bargmann transform to show that we have an explicit $\mf{k}$-finite decomposition of the Schr\"odinger representation $W_\lambda$.

\subsection{An intertwining operator}\label{SSIntOp}
\label{Subsection intertwiner}

\begingroup
\allowdisplaybreaks
We denote the Fock space $F_\lambda$ restricted to $\R$ by $F_{\lambda,\R}$ and the Fock representation $\rol$ restricted to $F_{\lambda,\R}$ by $\rho_{\lambda,\R}$. Define $C := \exp(\frac{\imath}{2}f)\exp(\imath e)$, with $e$ and $f$ as in Section \ref{Section Cayley}. Then the Cayley transform can be rewritten as
\begin{align*}
c &= \exp(\frac{\imath}{2}\ad(f))\exp(\imath\ad(e)) = \Ad(\exp(\frac{\imath}{2}f))\Ad(\exp(\imath e)) = \Ad(C),
\end{align*}
and therefore, we formally have,
\begin{align}\label{Eq Int}
\rho_{\lambda,\R}(X) &= \pil(c(X)) = \pil(\Ad(C)X) = \pil(C)\pil(X)\pil(C)^{-1}.
\end{align}

Equation (\ref{Eq Int}) shows that $\pil(C)$ intertwines the actions of $\pil$ and $\rho_{\lambda,\R}$. We wish to show that the space $W_\lambda$, defined in Section \ref{ssSchrod}, corresponds to $F_{\lambda, \R}$ under the transformation $\pil(C)$, i.e., we wish to show $\pil(C)^{-1}(F_{\lambda,\R}) = W_\lambda$. Let $x_1,x_2$ be the even representatives of the coordinates on $F_{\lambda,\R}$. Note that we have $\pil(C)^{-1} = \pil(C^{-1}) = \exp(\pil(-\imath e))\exp(\pi(-\dfrac{\imath}{2}f))$ with
\begin{align*}
\pil(-\dfrac{\imath}{2}f) &= -(x_1+x_2) \\
\pil(-\imath e) &= -\dfrac{1}{2}(\bessel(x_1)+\bessel(x_2))
\end{align*}
and recall that we are working modulo $\mc I_\lambda$. For $\lambda= \alpha$ we find
\begin{align*}
\iSB(1) &= \exp(-\dfrac{1}{2}(\bessel(x_1)+\bessel(x_2)))\exp(-x_1-x_2)\\
&= \exp(-\dfrac{1}{2}(-\alpha \pt {x_1} + x_1 \pt {x_1} ^2))\exp(-x_1)(1+\dfrac{1}{2}\pt {x_2})(1-x_2)\\
&= (\dfrac{1}{2}-x_2)\exp(-\dfrac{1}{2}(-\alpha  + \E)\pt {x_1})\exp(-x_1)\\
&=(\dfrac{1}{2}-x_2)\sum_{l=0}^\infty\sum_{j=0}^l \dfrac{(-1)^{l}}{j!(l-j)!2^j}(\alpha-l+1)_j {x_1}^{l-j}\\
&=(\dfrac{1}{2}-x_2)\sum_{i=0}^\infty (-1)^i \dfrac{\Gamma(\alpha-i+1)}{i!}\left( \sum_{j=0}^\infty \dfrac{(-1)^{j}}{2^{j}j! \Gamma(\alpha-i-j+1)}\right) x_1^{i}\\
&= (\dfrac{1}{2}-x_2)\sum_{i=0}^\infty (-1)^i \dfrac{\Gamma(\alpha-i+1)}{i!}\dfrac{2^{i-\alpha}}{\Gamma(\alpha-i+1)}x_1^i\\
&= 2^{-\alpha}(\dfrac{1}{2}-x_2)\exp(-2x_1)\\
&= \dfrac{1}{2^{1+\alpha}}\exp(-2(x_1+x_2)).
\end{align*}
Similarly, for $\lambda = 1$ we have \[\iSB(1) = \dfrac{1}{2^{1+\alpha^{-1}}}\exp(-2(x_1+x_2)).\]

We can now describe $W_\lambda$ explicitly using $\iSB$. If we calculate $\iSB(x_1^k)$ for $\lambda=\alpha$ we get
\begin{align*}
\iSB(x_1^k) &= \exp(-\dfrac{1}{2}(\bessel(x_1)+\bessel(x_2)))\exp(-x_1-x_2)x_1^k\\
&= (\dfrac{1}{2}-x_2)\exp(-\dfrac{1}{2}(-\alpha  + \E)\pt {x_1})\exp(-x_1)x_1^k\\
&= \dfrac{1}{2}\exp(-2x_2)\exp(-\dfrac{1}{2}(-\alpha  + \E)\pt {x_1})\exp(-x_1)x_1^k.
\end{align*}
In order to calculate this further we need the following lemma.

\begin{lemma}\label{LemSum}
We have
\begin{align*}
&((-\alpha+\ds E)\pt x)^j \left( \exp(-x)x^k  \right)\\
&\quad = \sum_{l=0}^\infty\sum_{i=0}^j\binom{i}{j}(-1)^{j-i}(-\alpha+l+2k-i)_{j-i}\dfrac{(-x)^l}{l!}((-\alpha+\E)\pt x)^ix^k,
\end{align*}
for all $x\in \R$ and $k,j\in \N$.
\end{lemma}

\begin{proof}
We will use induction on $j$. The cases $j=0$  and $j=1$ are obtained by a straightforward verification. Now suppose we have proven the induction hypothesis for $j\in \ds N$. We find
\begin{align*}
&((-\alpha+\ds E)\pt x)^{j+1} \exp(-x)x^k\\
& = (-\alpha+\ds E)\pt x\\
&\quad \left(\sum_{i=0}^j\sum_{l=0}^\infty\binom{i}{j}(-1)^{j-i}(-\alpha+l+2k-i)_{j-i}\dfrac{(-x)^l}{l!}((-\alpha+\E)\pt x)^ix^k\right)\\
& = (-\alpha+\ds E)\\
&\quad \left(\sum_{i=0}^j\binom{i}{j}\left(\sum_{l=0}^\infty(-1)^{j-i+1}(-\alpha+l+2k-i)_{j-i}\dfrac{(-x)^{l-1}}{(l-1)!}\right)((-\alpha+\E)\pt x)^ix^k\right)\\
&\quad + (-\alpha+\ds E)\\
&\quad \left(\sum_{i=0}^j\binom{i}{j}\left(\sum_{l=0}^\infty(-1)^{j-i}(-\alpha+l+2k-i)_{j-i}\dfrac{(-x)^l}{l!}\right)\pt x((-\alpha+\E)\pt x)^ix^k\right)\\
& = \sum_{i=0}^j\binom{i}{j}\left(\sum_{l=1}^\infty(-1)^{j-i+1}(-\alpha+l+k-i-1)(-\alpha+l+2k-i)_{j-i}\dfrac{(-x)^{l-1}}{(l-1)!}\right)\\
&\quad \times ((-\alpha+\E)\pt x)^ix^k\\
&\quad + \sum_{i=0}^j\binom{i}{j}\left(\sum_{l=0}^\infty l(-1)^{j-i}(-\alpha+l+2k-i)_{j-i}\dfrac{(-x)^l}{l!}\right)\pt x((-\alpha+\E)\pt x)^ix^k\\
&\quad + \sum_{i=0}^j\binom{i}{j}\left(\sum_{l=0}^\infty(-1)^{j-i}(-\alpha+l+2k-i)_{j-i}\dfrac{(-x)^l}{l!}\right)\\
&\quad \times(-\alpha+\ds E)\pt x((-\alpha+\E)\pt x)^ix^k\\
& = \sum_{i=0}^j\binom{i}{j}\left(\sum_{l=1}^\infty(-1)^{j-i+1}(-\alpha+l+k-i-1)(-\alpha+l+2k-i)_{j-i}\dfrac{(-x)^{l-1}}{(l-1)!}\right)\\
&\quad \times ((-\alpha+\E)\pt x)^ix^k\\
&\quad + \sum_{i=0}^j\binom{i}{j}\left(\sum_{l=1}^\infty(-1)^{j-i+1}(-\alpha+l+2k-i)_{j-i}\dfrac{(-x)^{l-1}}{(l-1)!}\right)\\
&\quad \times x\pt x((-\alpha+\E)\pt x)^ix^k\\
&\quad + \sum_{i=0}^j\binom{i}{j}\left(\sum_{l=0}^\infty(-1)^{j-i}(-\alpha+l+2k-i)_{j-i}\dfrac{(-x)^l}{l!}\right)\\
&\quad \times (-\alpha+\ds E)\pt x((-\alpha+\E)\pt x)^ix^k\\
\end{align*}
For the second sum we have
\begin{align*}
&\sum_{i=0}^j\binom{i}{j}\left(\sum_{l=0}^\infty l(-1)^{j-i}(-\alpha+l+2k-i)_{j-i}\dfrac{(-x)^l}{l!}\right)\pt x((-\alpha+\E)\pt x)^ix^k\\
&= \sum_{i=0}^j\binom{i}{j}\left(\sum_{l=1}^\infty(-1)^{j-i+1}(-\alpha+l+2k-i)_{j-i}\dfrac{(-x)^{l-1}}{(l-1)!}\right)x\pt x((-\alpha+\E)\pt x)^ix^k\\
&= \sum_{i=0}^j\binom{i}{j}\left(\sum_{l=1}^\infty(-1)^{j-i+1}(-\alpha+l+2k-i)_{j-i}\dfrac{(-x)^{l-1}}{(l-1)!}\right)(k-i)((-\alpha+\E)\pt x)^ix^k\\
\end{align*}
We can now rearrange the first and second sum to obtain
\begin{align*}
&((-\alpha+\ds E)\pt x)^{j+1} \exp(-x)x^k\\
& = \sum_{i=0}^j\binom{i}{j}\left(\sum_{l=1}^\infty(-1)^{j-i+1}(-\alpha+l+2k-i-1)(-\alpha+l+2k-i)_{j-i}\dfrac{(-x)^{l-1}}{(l-1)!}\right)\\
&\quad \times((-\alpha+\E)\pt x)^ix^k\\
&\quad + \sum_{i=1}^j\binom{i}{j}\left(\sum_{l=1}^\infty(-1)^{j+1-i}(-i)(-\alpha+l+2k-i)_{j-i}\dfrac{(-x)^{l-1}}{(l-1)!}\right)((-\alpha+\E)\pt x)^ix^k\\
&\quad + \sum_{i=0}^j\binom{i}{j}\left(\sum_{l=0}^\infty(-1)^{j-i}(-\alpha+l+2k-i)_{j-i}\dfrac{(-x)^{l}}{l!}\right)((-\alpha+\E)\pt x)^{i+1}x^k.
\end{align*}
We use the substitution $l\mapsto l+1$ in the first and second sum and the substitution $i\mapsto i-1$ in the third sum to find
\begin{align*}
&((-\alpha+\ds E)\pt x)^{j+1} \exp(-x)x^k\\
& = \sum_{i=0}^j\binom{i}{j}\left(\sum_{l=0}^\infty(-1)^{j-i+1}(-\alpha+l+2k-i)_{j+1-i}\dfrac{(-x)^{l}}{l!}\right)((-\alpha+\E)\pt x)^ix^k\\
&\quad + \sum_{i=1}^j\binom{i}{j}\left(\sum_{l=0}^\infty(-1)^{j+1-i}(-i)(-\alpha+l+1+2k-i)_{j-i}\dfrac{(-x)^{l}}{l!}\right)((-\alpha+\E)\pt x)^ix^k\\
&\quad + \sum_{i=1}^{j+1}\binom{i-1}{j}\left(\sum_{l=0}^\infty(-1)^{j-i+1}(-\alpha+l+2k-i+1)_{j-i+1}\dfrac{(-x)^{l}}{l!}\right)((-\alpha+\E)\pt x)^{i}x^k.
\end{align*}
If we combine the second and third sum we get
\begin{align*}
&((-\alpha+\ds E)\pt x)^{j+1} \exp(-x)x^k\\
& = \sum_{i=0}^j\binom{i}{j}\left(\sum_{l=0}^\infty(-1)^{j-i+1}(-\alpha+l+2k-i)_{j+1-i}\dfrac{(-x)^{l}}{l!}\right)((-\alpha+\E)\pt x)^ix^k\\
&\quad + \sum_{i=1}^j\sum_{l=0}^\infty(-1)^{j+1-i}(-\alpha+l+1+2k-i)_{j-i}\\
&\quad \times\left(\binom{i}{j}(-i)+\binom{i-1}{j}(-\alpha+l+2k+j-2i+1)\right)\dfrac{(-x)^{l}}{l!}((-\alpha+\E)\pt x)^ix^k\\
&\quad +\sum_{l=0}^\infty \dfrac{(-x)^l}{l!}\left((-\alpha+\E)\pt x\right)^{j+1}x,
\end{align*}
which becomes
\begin{align*}
&((-\alpha+\ds E)\pt x)^{j+1} \exp(-x)x^k\\
& = \sum_{i=0}^j\binom{i}{j}\left(\sum_{l=0}^\infty(-1)^{j-i+1}(-\alpha+l+2k-i)_{j+1-i}\dfrac{(-x)^{l}}{l!}\right)((-\alpha+\E)\pt x)^ix^k\\
&\quad + \sum_{i=1}^j\sum_{l=0}^\infty(-1)^{j+1-i}(-\alpha+l+1+2k-i)_{j-i}\\
&\quad \times\left(\binom{i-1}{j}(-\alpha+l+2k-i)\right)\dfrac{(-x)^{l}}{l!}((-\alpha+\E)\pt x)^ix^k\\
&\quad +\sum_{l=0}^\infty \dfrac{(-x)^l}{l!}\left((-\alpha+\E)\pt x\right)^{j+1}x^k\\
& = \sum_{i=0}^{j+1}\sum_{l=0}^\infty\binom{i}{j+1}(-1)^{j+1-i}(-\alpha+l+2k-i)_{j+1-i}\dfrac{(-x)^l}{l!}((-\alpha+\E)\pt x)^ix^k,
\end{align*}
as we wished to prove.
\end{proof}

Using this lemma we find
\begin{align*}
&\exp(-\dfrac{1}{2}(-\alpha  + \E)\pt {x})\exp(-x)x^k\\
&= \sum_{j=0}^\infty \dfrac{(-1)^j}{2^j j!}((-\alpha+\E)\pt x)^j(\exp(-x)x^k) \\
&= \sum_{j=0}^\infty \sum_{l=0}^\infty\sum_{i=0}^j\dfrac{(-1)^i}{2^j}\dfrac{1}{i!(j-i)!}(-\alpha+l+2k-i)_{j-i}\dfrac{(-x)^l}{l!}((-\alpha+\E)\pt x)^ix^k\\
&= \sum_{i=0}^\infty \sum_{j=i}^{\infty}\sum_{l=0}^\infty\dfrac{(-1)^i}{2^j}\dfrac{1}{i!(j-i)!}(-\alpha+l+2k-i)_{j-i}\dfrac{(-x)^l}{l!}((-\alpha+\E)\pt x)^ix^k\\
&= \sum_{i=0}^k \sum_{j=0}^{\infty}\sum_{l=0}^\infty\dfrac{(-1)^i}{2^{j+i}}\dfrac{1}{i!j!}(-\alpha+l+2k-i)_{j}\dfrac{(-x)^l}{l!}((-\alpha+\E)\pt x)^ix^k\\
&= \sum_{i=0}^k \sum_{l=0}^\infty\dfrac{(-1)^i}{2^{i}}\dfrac{1}{i!}2^{-\alpha+l+2k-i}\dfrac{(-x)^l}{l!}((-\alpha+\E)\pt x)^ix^k\\
&= \dfrac{1}{2^\alpha}\exp(-2x)\sum_{i=0}^k\dfrac{(-1)^i}{4^ii!}((-\alpha+\E)\pt x)^i(4x)^k\\
&= \dfrac{1}{2^\alpha}\exp(-2x)\sum_{i=0}^k \dfrac{(-1)^{i}}{i!}(\alpha-k+1)_i(-k)_i (4x)^{k-i}.
\end{align*}

If we now go back to what we originally wished to calculate, we find

\begin{align*}
\iSB(x_1^k) &= \dfrac{1}{2^{1+\alpha}}\exp(-2(x_1+x_2))\sum_{i=0}^k \dfrac{(-1)^{i}}{i!}(\alpha-k+1)_i(-k)_i (4x_1)^{k-i}.
\end{align*}

\subsection{The function $\KO_{\lambda,k}(x)$}

For $\lambda\in\{1,\alpha\}$ we define the polynomial $\KO_{\lambda,k}(x)$ by
\begin{align*}
\KO_{\alpha,k}(x) &:= \sum_{i=0}^k \dfrac{(-1)^{i}}{i!}(\alpha-k+1)_i(-k)_i (4x)^{k-i},\\
\KO_{1,k}(x) &:=  \sum_{i=0}^k \dfrac{(-1)^{i}}{i!}(\alpha^{-1}-k+1)_i(-k)_i (4x)^{k-i}.
\end{align*}
Then for $\lambda=\alpha$ the calculations in Section \ref{Subsection intertwiner} show
\begin{align*}
\iSB(x_1^k) &= \dfrac{1}{2^{1+\alpha}}\exp(-2(x_1+x_2))\KO_{\alpha,k}(x_1)
\end{align*}
and using similar calculations for $\lambda=1$ we find
\begin{align*}
\iSB(x_2^k) &= \dfrac{1}{2^{1+\alpha^{-1}}}\exp(-2(x_1+x_2))\KO_{1,k}(x_2).
\end{align*}

Note that the function $\KO_{\lambda,k}(x)$ can be given by the confluent hypergeometric function of the second kind, also known as Tricomi's confluent hypergeometric function and Kummer's function of the second kind (see, e.g., \cite[Chapter 13]{AS}). For $a,b,c\in \C$, we define the confluent hypergeometric function of the second kind $\KU(a,b,c)$ as
\begin{align*}
\KU(a,b,c)  := c^{-a}\HypFdegen(a,1+a-b, -c^{-1}),
\end{align*}
where
\begin{align*}
\HypFpq (a_1,\ldots, a_p, b_1, \ldots, b_q,c) := \sum_{i=0}^\infty \dfrac{(a_1)_i\cdots (a_p)_i}{(b_1)_i\cdots (b_q)_i}\dfrac{c^i}{i!}
\end{align*}
denotes the generalized hypergeometric function. The function $\KU(a,b,c)$ is a solution of
\begin{align} \label{EqDiffKummer}
c\pt {c}^2 u +(b-c) \pt {c} u - au = 0,
\end{align}
which is known as Kummer's differential equation.

\begin{Prop}\label{EqKummer}
We have
\begin{align*}
\KO_{\alpha,k}(x_1) = \KU(-k,-\alpha,4x_1),  \quad \text{ and } \quad \KO_{1,k}(x_2) = \KU(-k,-\alpha^{-1},4x_2),
\end{align*}
for all $k\in\N$.
\end{Prop}

\begin{proof}
We can easily prove this proposition by filling in the correct parameters in the definition of $\KU(a,b,c)$. However, the proof below gives us a more insightful reason as to why this proposition holds.

We prove the $\lambda=\alpha$ case. The $\lambda=1$ is entirely analoguous. On the one hand we have 
\begin{align}\label{EqKO}
\iSB(x_1^k) = \dfrac{1}{2^{1+\alpha}}\exp(-2(x_1+x_2))\KO_{\alpha,k}(x_1) 
\end{align}
and since $\iSB$ intertwines $\pil$ and $\rho_{\lambda,\R}$, we also have  on the other hand 
\begin{align*}
\iSB(x_1^k) &= \dfrac{-\imath}{\alpha-2k}\iSB(\imath(\alpha-2\E)x_1^{k})\\
&= \dfrac{-\imath}{\alpha-2k}\iSB(\rho_{\lambda,\R}(f_1,0, -e_1)x_1^{k}) \\
&= \dfrac{-\imath}{\alpha-2k}\pil(f_1,0, -e_1)\iSB(x_1^{k})\\
&= \dfrac{-1}{2(\alpha-2k)}\left(4x_1-\bessel(x_1)\right)\dfrac{1}{2^{1+\alpha}}\exp(-2(x_1+x_2))\KO_{\alpha,k}(x_1).
\end{align*}
Therefore
\begin{align*}
2(2k-\alpha)\exp(-2x_1)\KO_{\alpha,k}(x_1) &= 4x_1 \exp(-2x_1)\KO_{\alpha,k}(x_1) +\alpha \pt {x_1} \exp(-2x_1)\KO_{\alpha,k}(x_1)\\
&\quad - x_1 \pt {x_1}^2\exp(-2x_1)\KO_{\alpha,k}(x_1),
\end{align*}
which is equivalent to
\begin{align*}
2(2k-\alpha)\KO_{\alpha,k}(x_1) &= 4x_1 \KO_{\alpha,k}(x_1) +\alpha  \pt {x_1}\KO_{\alpha,k}(x_1) -2\alpha \KO_{\alpha,k}(x_1) - 4 x_1 \KO_{\alpha,k}(x_1)\\
&\quad + 4 x_1 \pt {x_1}\KO_{\alpha,k}(x_1)- x_1 \pt {x_1}^2\KO_{\alpha,k}(x_1),
\end{align*}
which further simplifies to
\begin{align}\label{EqDiff}
x_1\pt {x_1}^2 \KO_{\alpha,k}(x_1) + (-\alpha - 4x_1)\pt {x_1}\KO_{\alpha,k}(x_1)+4k \KO_{\alpha,k}(x_1) =0.
\end{align}
This implies $\KO_{\alpha,k}(x_1)$ is a solution to equation (\ref{EqDiffKummer}) for $a = -k$, $b = -\alpha$ and $c=4x_1$. Solving this differential equation using the initial conditions 
\begin{align*}
\KO_{\alpha, k}(0) &=(-1)^k(-\alpha)_k,\\
(\pt {x_1} \KO_{\alpha, k})(0) &= 4k(-1)^k(1-\alpha)_{k-1},
\end{align*}
we obtain $\KO_{\alpha,k}(x_1) = \KU(-k,-\alpha,4x_1)$.
\end{proof}

For $\lambda\in\{1,\alpha\}$ we also define the polynomial $\KV_{\lambda,k}(x)$ by
\begin{align*}
\KV_{\alpha,k}(x) &:= \KU(-k,1-\alpha,4x),\\
\KV_{1,k}(x) &:=  \KU(-k,1-\alpha^{-1},4x).
\end{align*}
Then, for a polynomial of degree $k$ in $F_{\lambda,\R}$ we now have the following result.

\begin{theorem}\label{CorMon}
Let $p = (p_1 x_1 + p_2 x_2 + p_3 {x_3} + p_4 {x_4})x_i^{k-1}$ be a polynomial of degree $k\in \ds N$ in $F_{\lambda,\R}$ where $i=1$ if $\lambda=\alpha$ and $i=2$ if $\lambda=1$. We have
\begin{align*}
\iSB(p) &= \dfrac{1}{2^{1+\alpha}}p_1 \KO_{\alpha,k}(x_1)\exp(-2(x_1+x_2))\\
&\quad + \dfrac{1}{2^{1+\alpha}}p_2 \KO_{\alpha,k-1}(x_1)\exp(-2(x_1-x_2))\\
&\quad + \dfrac{1}{2^{\alpha-1}}(p_3 {x_3} + p_4 {x_4})\KV_{\alpha,k-1}(x_1)\exp(-2x_1).
\end{align*}
for $\lambda=\alpha$ and
\begin{align*}
\iSB(p) &= \dfrac{1}{2^{1+\alpha^{-1}}}p_2 \KO_{1,k}(x_2)\exp(-2(x_1+x_2))\\
&\quad + \dfrac{1}{2^{1+\alpha^{-1}}}p_1 \KO_{1,k-1}(x_2)\exp(-2(x_2-x_1))\\
&\quad + \dfrac{1}{2^{\alpha^{-1}-1}}(p_3 {x_3} + p_4 {x_4})\KV_{1,k-1}(x_2)\exp(-2x_2).
\end{align*}
for $\lambda=1$.
\end{theorem}

\begin{proof}
This follows directly from the calculations in Section \ref{SSIntOp} and calculations easily derived from them. For example, if $\lambda = \alpha$ we have $\exp(-x_2) x_3 \equiv x_3$  modulo $\mc I_\lambda$ and thus we find using the commutation relation $[\E,x_3]=x_3$ 
\begin{align*}
\iSB(x_1^{k-1}x_3) &= \exp(-\dfrac{1}{2}(-\alpha + \E)\pt {x_1})\exp(-x_1)x_1^{k-1}x_3\\
&= x_3\exp(-\dfrac{1}{2}(-\alpha +1 + \E)\pt {x_1})\exp(-x_1)x_1^{k-1}.
\end{align*}
Therefore, we obtain from the calculations below Lemma \ref{LemSum}
\begin{align*}
\iSB(x_1^{k-1}x_3) &= x_3\dfrac{1}{2^{\alpha-1}}\exp(-2x_1)\KU(-k+1,-\alpha+1,4x_1)\\
&= x_3\dfrac{1}{2^{\alpha-1}}\exp(-2x_1)\KV_{\alpha,k-1}(x_1),
\end{align*}
as expected.
\end{proof}

\endgroup
\subsection{The Segal-Bargmann transform}\label{Segal-Bargmann transform}

In order to construct an intertwining operator from $W_\lambda$ to $F_\lambda$, we modify $\SB$ using the reproducing kernel $\ds I_\lambda(z,w)$ given in Theorem \ref{Theorem repr kernel}.

\begin{Def}\label{Def SB}
 For $f\in W_\lambda$ the \textbf{Segal-Bargmann transform} is defined as
\begin{align*}
\rSB (f(x))(z) &:= \bfipx{\SB(f(x))(x), \ds I_\lambda(z,x)}\\
&= \bfipx{\exp(x_1+x_2)\exp(\dfrac{1}{2}(\bessel(x_1)+\bessel(x_2)))f,\ds I_\lambda(z,x)},
\end{align*}
where $\bfipx{\cdot\, , \cdot}$ denotes the Bessel-Fischer product in the variable $x$.
\end{Def}

For $p\in F_\lambda$ the inverse Segal-Bargmann transform is then given by
\begin{align*}
\irSB (p(z))(x) &= \iSB(\bfip{p, \ds I_\lambda(z,x)})(x)\\
&= \exp(-\dfrac{1}{2}(\bessel(x_1)+\bessel(x_2)))\exp(-x_1-x_2)\bfip{p,\ds I_\lambda(z,x)}.
\end{align*}
Because of Theorem \ref{Theorem repr kernel} and the way $\SB$ is defined we immediately have the intertwining property.

\begin{theorem}[Intertwining property]\label{ThIP}
Suppose either $\lambda=\alpha$ and $\alpha\not\in \N$ or $\lambda=1$ and $\alpha^{-1}\not\in \N$. The Segal-Bargmann transform intertwines the action $\pil$ on $W_\lambda$ with the action $\rol$ on $F_\lambda$, i.e.,
\begin{align*}
\rSB\circ\, \pil(X) = \rol(X)\circ \rSB,
\end{align*}
for all $X\in \mathfrak{g}$.
\end{theorem}

Because there exists a well-defined inverse Segal-Bargmann transform we also have the following property.

\begin{Prop} \label{Prop inverse SB transform}
Suppose either $\lambda=\alpha$ and $\alpha\not\in \N$ or $\lambda=1$ and $\alpha^{-1}\not\in \N$. The Segal-Bargmann transform $\rSB$ induces a $\mf g$-module isomorphism between $W_\lambda$ and $F_\lambda$.
\end{Prop}

We can now construct the analogue of Theorem \ref{ThDecF} for the $(\mf g, \mf k)$-module $W_\lambda$. For $k\geq 1$ we introduce
\begin{align*}
W_{\lambda,0} &:= \irSB(F_{\lambda,0}) = \C\exp(-2(x_1+x_2)),\\
W_{\lambda,k} &:= \irSB(F_{\lambda,k}) = \C \KO_{\lambda,k}(x_i)\exp(-2(x_1+x_2))\\
&\quad + \C \KO_{\lambda,k-1}(x_i)\exp(-2(x_i-x_j)) +(\C x_3 + \C x_4)\KV_{\lambda,k-1}(x_i)\exp(-2x_i),\\
K_{\alpha, k} &:= \irSB(H_{\alpha,k}) = \C (\KO_{\alpha,k}(x_1)\exp(-2(x_1+x_2))\\
&\quad +(k-1-\alpha)\KO_{\alpha,k-1}(x_1)\exp(-2(x_1-x_2)) \\
&\quad + (\C x_3 + \C x_4)\KV_{\alpha,k-1}(x_1)\exp(-2x_1),\\
K_{1, k} &:= \irSB(H_{1,k}) = \C (\KO_{1,k}(x_2)\exp(-2(x_1+x_2))\\
&\quad +(k-1-\alpha^{-1})\KO_{1,k-1}(x_2)\exp(-2(x_2-x_1)) ) \\
&\quad + (\C x_3 + \C x_4)\KV_{1,k-1}(x_2)\exp(-2x_2),\\
R_{\lambda,k} &:= \irSB((z_1+z_2)^k) = \C (\KO_{\lambda,k}(x_i)\exp(-2(x_1+x_2)) \\
&\quad +k\KO_{\lambda,k-1}(x_i)\exp(-2(x_i-x_j)) ),
\end{align*}
where $(i,j)=(1,2)$ if $\lambda=\alpha$ and $(i,j)=(2,1)$ if $\lambda=1$. Combining Proposition \ref{Prop inverse SB transform} with Theorem \ref{ThDecF} we get the following theorem.

\begin{theorem}[Decomposition of $W$]\label{ThDecW}
Suppose either $\lambda=\alpha$ and $\alpha\not\in \N$ or $\lambda=1$ and $\alpha^{-1}\not\in \N$. We have the following:
\begin{itemize}
\item[(1)] \begin{itemize}
\item[(a)] For $\alpha\neq -1$ an explicit decomposition of $W_{\lambda,k}$ into irreducible $\mathfrak{k}_{0}$-modules is given by
\begin{align*}
W_{\lambda,k} = K_{\lambda,k} \oplus R_{\lambda,k}.
\end{align*}
\item[(b)] If $\alpha=-1$, then $W_{\lambda,k}$ is an indecomposable $\mf k_0$-module, but not an irreducible $\mf k_0$-module.
\end{itemize}
\item[(2)]$W_{\lambda,k}$ is an irreducible $\mf k$-module.
\item[(3)]$W_\lambda$ is an irreducible $\mathfrak{g}$-module and its $\mf k$-type decomposition is given by
\begin{align*}
W_{\lambda} = \bigoplus_{k=0}^\infty W_{\lambda,k}.
\end{align*}
\end{itemize}
\end{theorem}

We can construct a non-degenerate supersymmetric sesquilinear form on $W_\lambda$ using the Segal-Bargmann transform in conjunction with the Bessel-Fischer product. More specifically, if we define
\begin{align*}
\ip{f,g} := \bfip{\rSB f, \rSB g},
\end{align*}
for all $f,g\in W_\lambda$, then $\ip{\cdot \, , \cdot}$ is a non-degenerate supersymmetric sesquilinear form on $W_\lambda$. By definition we now have the following property.

\begin{theorem}[Unitary property]\label{PropUnitSB}
Suppose either $\lambda=\alpha$ and $\alpha\not\in \N$ or $\lambda=1$ and $\alpha^{-1}\not\in \N$. The Segal-Bargmann transform preserves the sesquilinear forms, i.e.,
\begin{align*}
\bfip{\rSB f,\rSB g} = \ip{f, g},
\end{align*}
for all $f,g\in W_\lambda$.
\end{theorem}

\subsection{Recurrence relations of $\KO_{\lambda,k}$}\label{SecOmega}
We finish this paper by taking a closer look at the polynomials $\KO_{\lambda,k}$. Using the intertwining property of $\SB$ we can deduce several relations between these polynomials. For example in this section we recover two known differential recurrence relations and one regular recurrence relations for the confluent hypergeometric functions. 

We consider only the case $\lambda=\alpha$. The $\lambda =1$ case is obtained by substituting the occurrences of $\alpha$ in the expressions with $\alpha^{-1}$ and switching the roles of $x_1$ and $x_2$.

Combining equation (\ref{EqKO}) with
\begin{align*}
\iSB(x_1^k) &= \dfrac{\imath}{2}\iSB(\rol(f_1,-2\imath L_{e_1}, e_1)x_1^{k-1}) \\
&=  \dfrac{\imath}{2}\pil(f_1,-2\imath L_{e_1}, e_1)\iSB(x_1^{k-1})\\
&= \dfrac{1}{2^{1+\alpha}}\left(x_1+\dfrac{\alpha}{2}-\E+\dfrac{1}{4}\bessel(x_1)\right)\KO_{\alpha, k-1}(x_1)\exp(-2x_2)
\end{align*}
we get
\begin{align*}
4\KO_{\alpha, k}(x_1)\exp(-2x_1) &= 4\left(x_1+\dfrac{\alpha}{2}-\E+\dfrac{1}{4}\bessel(x_1)\right)\KO_{\alpha, k-1}(x_1)\exp(-2x_1)\\
&= \left(4x_1+2\alpha\right)\KO_{\alpha, k-1}(x_1)\exp(-2x_1)\\ 
&\quad - \left(4x_1+\alpha\right)\pt {x_1}\KO_{\alpha, k-1}(x_1)\exp(-2x_1)\\
&\quad + x_1 \pt {x_1}^2 \KO_{\alpha, k-1}(x_1)\exp(-2x_1).
\end{align*}
Therefore we have the differential recurrence relation
\begin{align*}
4\KO_{\alpha, k}(x_1) &= \left(4x_1+2\alpha\right)\KO_{\alpha, k-1}(x_1)\\ 
&\quad - \left(4x_1+\alpha\right)\pt {x_1}\KO_{\alpha, k-1}(x_1) + 2\left(4x_1+\alpha\right)\KO_{\alpha, k-1}(x_1)\\
&\quad + x_1 \pt {x_1}^2\KO_{\alpha, k-1}(x_1) - 4 x_1 \pt {x_1} \KO_{\alpha, k-1}(x_1) + 4 x_1 \pt {x_1}^2 \KO_{\alpha, k-1}(x_1)\\
&= 4\left(4x_1+\alpha\right)\KO_{\alpha, k-1}(x_1) - \left(8x_1+\alpha\right)\pt {x_1}\KO_{\alpha, k-1}(x_1)\\
&\quad + x_1 \pt {x_1}^2\KO_{\alpha, k-1}(x_1).
\end{align*}
Taking equation (\ref{EqDiff}) into account this becomes
\begin{align}\label{Diffrec1}
\KO_{\alpha, k}(x_1) &= \left(4x_1 + \alpha - k + 1\right)\KO_{\alpha, k-1}(x_1) - x_1\pt {x_1}\KO_{\alpha, k-1}(x_1).
\end{align}
Similarly, if we combine equation (\ref{EqKO}) with
\begin{align*}
\iSB(x_1^k) &= \dfrac{\imath}{2(k+1)(k-\alpha)}\iSB(\rol(f_1,2\imath L_{e_1}, e_1)x_1^{k+1}) \\
&=  \dfrac{\imath}{2(k+1)(k-\alpha)}\pil(f_1,2\imath L_{e_1}, e_1)\iSB(x_1^{k+1})\\
&= \dfrac{1}{2^{1+\alpha}(k+1)(k-\alpha)}\\
&\quad \times\left(x_1+\dfrac{\alpha}{2}-\E+\dfrac{1}{4}\bessel(x_1)\right)\KO_{\alpha, k+1}(x_1)\exp(-2(x_1+x_2))
\end{align*}
we get the differential recurrence relation
\begin{align*}
4(k+1)(k-\alpha)\KO_{\alpha, k}(x_1) &=  -\alpha\pt {x_1}\KO_{\alpha, k+1}(x_1) + x_1 \pt {x_1}^2 \KO_{\alpha, k+1}(x_1).
\end{align*}
Taking equation (\ref{EqDiff}) into account this becomes
\begin{align}\label{Diffrec2}
(k+1)(k-\alpha)\KO_{\alpha, k}(x_1) &= -\left(k + 1\right)\KO_{\alpha, k+1}(x_1) + x_1\pt {x_1}\KO_{\alpha, k+1}(x_1).
\end{align}
We can rewrite relations (\ref{Diffrec1}) and (\ref{Diffrec2}) as
\begin{align*}
x_1\pt {x_1} \KO_{\alpha,k}(x_1) &= \left(4x_1+\alpha-k \right)\KO_{\alpha,k}(x_1) - \KO_{\alpha,k+1}(x_1),\\
x_1\pt {x_1} \KO_{\alpha,k}(x_1) &= k\KO_{\alpha,k}(x_1) - k(\alpha-k+1)\KO_{\alpha,k-1}(x_1),
\end{align*}
respectively. Subtracting the first equation from the second gives us the recurrence relation
\begin{align*}
\KO_{\alpha,k+1}(x_1) + (2k-\alpha-4x_1)\KO_{\alpha,k}(x_1)-k(\alpha-k+1)\KO_{\alpha,k-1}(x_1) = 0.
\end{align*}
By using Proposition \ref{EqKummer} we can write the above relations as recurrence relations for the confluent hypergeometric function of the second kind,
\begin{align*}
x\pt {x} U(-k,-\alpha,x) &= \left(x+\alpha-k \right)U(-k,-\alpha,x) - U(-k-1,-\alpha,x),\\
x\pt {x} U(-k,-\alpha,x) &= kU(-k,-\alpha,x) - k(\alpha-k+1)U(-k+1,-\alpha,x),
\end{align*}
and
\begin{align*}
(x-2k+\alpha)\KU(-k,-\alpha,x)=\KU(-k-1,-\alpha,x)-k(\alpha-k+1)\KU(-k+1,-\alpha,x),
\end{align*}
with $k\in\N\setminus \{0\}$, $\alpha\not\in \N$ and $x\in \R$. These relations are known to hold more generally and can be obtained as combinations of \cite[Equation 13.3.7]{NIST:DLMF}, \cite[Equation 13.3.10]{NIST:DLMF} and  \cite[Equation 13.3.22]{NIST:DLMF}.

\subsection*{Acknowledgements}
SB is supported by a FWO postdoctoral junior fellowship from the Research Foundation Flanders (1269821N).
The authors thank Hendrik De Bie and Jan Frahm for interesting and useful discussions.

\bibliography{citations} 
\bibliographystyle{ieeetr}

\end{document}